\documentclass[11pt, oneside]{amsart}
\usepackage{geometry}
\usepackage{graphicx}
\usepackage{amssymb,amsmath,amsfonts}
\usepackage[utf8]{inputenc}
\usepackage{mathtools,mathdots,mathabx}
\usepackage[all]{xy}

\usepackage{jakelev}

\newtheorem{condition}[theorem]{Condition}

\usepackage[bookmarks, colorlinks=true, linkcolor=blue, citecolor=blue, urlcolor=blue]{hyperref}
\usepackage[all]{hypcap}

\newcommand{\Emb}{\mathrm{Emb}}
\newcommand{\rect}{{\scalebox{.3}{\yng(3,3)}}}
\renewcommand{\phi}{\varphi}
\newcommand{\ann}{\mr{ann}}
\newcommand{\eylat}{\mb{Y}_{\pm}} 
\newcommand{\ercone}[1]{\widetilde{BS}{}^D_{#1}}
\newcommand{\emcone}[1]{BS^D_{#1}}
\newcommand{\rcone}[1]{\widetilde{BS}_{#1}}
\newcommand{\mcone}[1]{BS_{#1}}
\newcommand{\escone}[1]{ES_{#1}}

\newcommand{\shiftedpuretable}[3]{\widetilde{\beta}\left[{#2} \xleftarrow{{#1\ \ }} {#3} \right]}

\title{Foundations of Boij-S\"{o}derberg Theory for Grassmannians}
\author{Nic Ford and Jake Levinson}

\begin{document}

\begin{abstract}
Boij-S\"{o}derberg theory characterizes syzygies of graded modules and sheaves on projective space. This paper continues earlier work with S. Sam \cite{FLS16}, extending the theory to the setting of $GL_k$-equivariant modules and sheaves on Grassmannians. Algebraically, we study modules over a polynomial ring in $k n$ variables, thought of as the entries of a $k \times n$ matrix.

We give equivariant analogues of two important features of the ordinary theory: the Herzog-K\"{u}hl equations and the pairing between Betti and cohomology tables.
As a necessary step, we also extend the result of \cite{FLS16}, concerning the base case of square matrices, to cover complexes other than free resolutions.

Our statements specialize to those of ordinary Boij-S\"{o}derberg theory when $k=1$. Our proof of the equivariant pairing gives a new proof in the graded setting: it relies on finding perfect matchings on certain graphs associated to Betti tables.

Finally, we give preliminary results on $2 \times 3$ matrices, exhibiting certain classes of extremal rays on the cone of Betti tables.
\end{abstract}

\maketitle

\section{Introduction}

\subsection{Boij-S\"{o}derberg theory}
Let $R = \mb{C}[x_1, \ldots, x_n]$ be a polynomial ring and $M$ a graded, finitely-generated $R$-module. The \newword{Betti table} of $M$ is the collection of numbers
\[\beta_{ij}(M) := \#\text{ degree-$j$ generators of the $i$-th syzygy module of } M.\]
That is, by definition, the minimal free resolution of $M$ has the form
\[M \leftarrow F_0 \leftarrow F_1 \leftarrow \cdots \leftarrow F_n \leftarrow 0, \text{ with } F_i = \bigoplus_j R[-j]^{\beta_{ij}}.\]
These numbers encode much of the structure of $M$, such as its dimension and whether or not it is Cohen-Macaulay. They also describe geometrical properties of the associated sheaf on $\mb{P}(\mb{C}^n)$.

Boij-S\"{o}derberg theory was initially concerned with describing which sets of numbers $\beta_{ij}$ arise as Betti tables of modules. The key early observation was that it is easier to determine which tables arise up to scalar multiple, and as such, the goal was to characterize the \newword{Boij-S\"{o}derberg cone} $BS_n$ of positive scalar multiples of Betti tables \cite{BS2008,ES2009,EFW2011}. More recent work has focused on modules over multigraded and toric rings \cite{EE2012}, and some homogeneous coordinate rings \cite{BBEG2012,3pts,kummini-sam}, as well as more detailed homological questions \cite{NS12,BEKS13,EES13}. A good survey of the field is \cite{floystad-expository}.

In each of these cases, an important feature of the theory is a duality between Betti tables and \newword{cohomology tables} of sheaves on an associated variety. \cite{ES2009,EE2012} For a coherent sheaf $\mc{E}$ on $\mb{P}(\mb{C}^n)$, the cohomology table is the collection of numbers
\[\gamma_{ij}(\mc{E}) := \dim_\mb{C} H^i(\mc{E}(-j)),\]
giving all the sheaf cohomology of all the twists of $\mc{E}$.

For graded modules and projective spaces, the duality takes the form of a bilinear pairing of the cones of Betti and cohomology tables, and it produces a point in the simplest Boij-S\"oderberg cone $BS_1$. The inequalities defining $BS_1$ therefore pull back to nonnegative bilinear pairings between Betti and cohomology tables, and these pulled-back inequalities fully characterize the two cones.

In particular, the Boij-S\"{o}derberg cone (in the graded setting) is rational polyhedral, and its extremal rays and supporting hyperplanes are explicitly known. The extremal rays correspond to \newword{pure Betti tables}. These are the simplest possible tables, having only one nonzero entry in each column (that is, for each $i$, only one $\beta_{ij}$ is nonzero). Similarly, the supporting hyperplanes come from pairing with the cohomology of vector bundles with so-called \newword{supernatural cohomology}. Much less is known about the Boij-S\"{o}derberg cone in other settings.

\subsection{Grassmannian Boij-S\"{o}derberg theory} \label{subsec:grass-boijsoderberg}
The goal of this paper is to continue earlier work of the authors, joint with S. Sam \cite{FLS16}, on extending the theory to the setting of Grassmannians $Gr(k,\mb{C}^n)$. On the geometric side, we will be interested in the cohomology of coherent sheaves on $Gr(k,\mb{C}^n)$. On the algebraic side, we consider the polynomial ring in $kn$ variables ($k \leq n$),
\[R_{k,n} = \mb{C}\Big[ x_{ij} : \begin{aligned} &1 \leq i \leq k \\ &1 \leq j \leq n\end{aligned} \Big],\]
thought of as the entries of a $k \times n$ matrix. The group $GL_k$ acts on $R_{k,n}$, and we are interested in (finitely-generated) \newword{equivariant modules} $M$, i.e., those with a compatible $GL_k$ action.

Aside from the inherent interest of understanding sheaf cohomology and syzygies on Grassmannians, there is hope that this setting might avoid some obstacles faced in other extensions of Boij-S\"{o}derberg theory, e.g. to products of projective spaces. For example, in the `base case' of square matrices $(n=k)$, the `irrelevant ideal' is the principal ideal generated by the determinant, and the Boij-S\"{o}derberg cone has an especially elegant structure (see below, Section \ref{subsubsec:base-case}). 

We define \newword{equivariant Betti tables} $\beta(M)$ using the representation theory of $GL_k$. Let $\mb{S}_\lambda(\mb{C}^k)$ denote the irreducible $GL_k$ representation of weight $\lambda$, where $\mb{S}_\lambda$ is the Schur functor. There is a corresponding free module, namely $\mb{S}_\lambda(\mb{C}^k) \otimes_\mb{C} R_{k,n}$, and every equivariant free module is a direct sum of these. Then $\beta(M)$ is the collection of numbers
\[\beta_{i,\lambda}(M) := \# \text{ copies of } \mb{S}_\lambda(\mb{C}^k) \text{ in the generators of the $i$-th syzygy module of } M.\]
Thus, by definition, the minimal \emph{equivariant} free resolution of $M$ has the form
\[M \leftarrow F_0 \leftarrow \cdots \leftarrow F_n \leftarrow 0, \text{ with } F_i = \bigoplus_\lambda \mb{S}_\lambda(\mb{C}^k)^{\beta_{i\lambda}} \otimes R_{k,n}.\]
Next, for $\mc{E}$ a coherent sheaf on $Gr(k,\mb{C}^n)$, we will define the \newword{$\mathbf{GL}$-cohomology table} $\gamma(\mc{E})$, generalizing the usual cohomology table:
\[\gamma_{i,\lambda}(\mc{E}) := \dim H^i(\mc{E} \otimes \mb{S}_\lambda(\mc{S})),\]
where $\mc{S}$ is the tautological vector bundle on $Gr(k,\mb{C}^n)$ of rank $k$. We write $\mb{BT}_{k,n} := \bigoplus_{i,\lambda} \mb{Q}$ for the space of abstract Betti tables. Similarly, we write $\mb{CT}_{k,n} := \bigoplus_i \prod_\lambda \mb{Q}$ for the space of abstract $GL$-cohomology tables.

\begin{remark} The case $k=1$ reduces to the ordinary Boij-S\"{o}derberg theory, since an action of $GL_1$ is formally equivalent to a grading; the module $R[-j]$ is just $\mb{S}_{(j)}(\mb{C}) \otimes R$. Note also that $\mc{S} = \mc{O}(-1)$ on projective space.
\end{remark}
The initial questions of Boij-S\"{o}derberg theory concerned finite-length graded modules $M$, i.e. those annihilated by a power of the homogeneous maximal ideal, and more generally Cohen-Macaulay modules. Similarly, we restrict our focus (for now!) on the following class of modules, which specializes to finite-length modules when $k=1$:
\begin{condition}[The modules of interest] \label{cond:modules-of-interest}
We consider finitely-generated Cohen-Macaulay modules $M$ such that $\sqrt{\ann(M)} = P_k$, the ideal of maximal minors of the $k \times n$ matrix. \end{condition}
Viewing $\Spec(R_{k,n}) = \Hom(\mb{C}^k,\mb{C}^n)$ as the affine variety of $k \times n$ matrices, this means $M$ is set-theoretically supported on the locus of rank-deficient matrices. That is, the sheaf associated to $M$ on $Gr(k,\mb{C}^n)$ is zero. For this reason, we refer to $P_k$ as the \newword{irrelevant ideal} for this setting. The Cohen-Macaulayness assumption means that
\[\mathrm{pdim}(M) = \dim(R_{P_k}) = n-k+1,\]
so its minimal free resolution has length $n-k+1$.

\begin{definition}
We define the \newword{equivariant Boij-S\"{o}derberg cone} $\mcone{k,n} \subset \mb{BT}_{k,n}$ as the positive linear span of Betti tables $\beta(M)$, where $M$ satisfies the assumptions of Condition \ref{cond:modules-of-interest}. We define the \newword{Eisenbud-Schreyer cone} $\escone{k,n} \subset \mb{CT}_{k,n}$ as the positive linear span of $GL$-cohomology tables of all coherent sheaves $\mc{E}$ on $Gr(k,\mb{C}^n)$.
\end{definition}

We wish to understand the cones $\mcone{k,n}$ and $\escone{k,n}$ generated by equivariant Betti tables of and $GL$-cohomology tables.

\begin{remark}[Multiplicities and ranks]
The irreducible representations $\mb{S}_\lambda(\mb{C}^k)$ need not be one-dimensional. As such, the corresponding free modules need not have rank 1. We will write a tilde $\widetilde{\beta}$ to denote the \emph{rank} of the $\lambda$-isotypic component (rather than its \emph{multiplicity}), and likewise write $\widetilde{\mb{BT}}_{k,n}$ and $\widetilde{BS}_{k,n}$ for the spaces of rank Betti tables. 
Of course, we may switch between them by rescaling each entry $\beta_{i,\lambda}$ by $\dim(\mb{S}_\lambda(\mb{C}^k))$.
\end{remark}

\subsection{Results of this paper}

We will generalize two important results from the existing theory on graded modules: the Herzog-K\"{u}hl equations and the pairing between Betti and cohomology tables. Along the way, we also extend our existing result on equivariant modules for the square matrices.

\subsubsection{Equivariant Herzog-K\"{u}hl equations} \label{subsubsec:hk-eqns}
In the graded setting, the Herzog-K\"{u}hl equations are $n$ linear conditions satisfied by the Betti tables of finite-length modules $M$. They say, essentially, that the Hilbert polynomial of $M$ vanishes identically, i.e., each of its coefficients is zero.

We give the following equivariant analogue.
\begin{theorem} \label{thm:equi-hk}
Let $M$ be an equivariant $R_{k,n}$-module with Betti table $\beta(M)$. Assume (twisting up if necessary) that $M$ is generated in positive degree.

There is a system of $\binom{n}{k}$ linear conditions on $\beta(M)$, indexed by partitions $\mu \geq 0$ that fit inside a $k \times (n-k)$ rectangle, called the \newword{equivariant Herzog-K\"{u}hl equations}. The following are equivalent:
\begin{enumerate}
\item[(i)] $\beta(M)$ satisfies the equivariant Herzog-K\"{u}hl equations.
\item[(ii)] $M$ is annihilated by a power of the ideal $P_k$ of maximal minors.
\item[(iii)] The sheaf associated to $M$ on $Gr(k,\mb{C}^n)$ vanishes.
\end{enumerate}
In particular, the hypotheses of Condition \ref{cond:modules-of-interest} are equivalent to the equivariant Herzog-K\"{u}hl equations, together with the conditions $\beta_{i,\lambda} = 0$ for all $i > n-k+1$.
\end{theorem}
We state the equations in Section \ref{subsec:FLSHK}, using the combinatorics of standard Young tableaux. We also give an interpretation in terms of equivariant K-theory, namely that the class of $M$ lies in the kernel of the map of K-theory rings
\[K^{GL_k}(\Spec(R_{k,n})) \to K(Gr(k,\mb{C}^n))\]
induced by restriction (to the locus of full-rank matrices) and descent.

\subsubsection{The Boij-S\"{o}derberg cone for square matrices} \label{subsubsec:base-case}

As in the ordinary theory, we expect the smallest case $n=k$ to play an important role. 
It serves as the base case of the theory and the target of the equivariant Boij-S\"{o}derberg pairing (Section \ref{subsubsec:equi-pairing}). Rank tables $\widetilde{\beta}$ turn out to be more significant here, so we will state results in terms of the cone $\rcone{k,k}$.

For square matrices, the modules of interest are Cohen-Macaulay and have $\sqrt{\mr{ann}(M)} = (\det)$, so they have projective dimension 1. The cone $\rcone{k,k}$ is fully understood:
\begin{theorem}[{\cite[Theorem 1.2]{FLS16}}]
The cone $\rcone{k,k}$ is rational polyhedral. Its supporting hyperplanes are indexed by {\bf order ideals} in the extended Young's lattice $\mb{Y}_{\pm}$ of $GL_k$-representations. Its extremal rays are indexed by {\bf comparable pairs} $\lambda\subsetneq \mu$ from $\mb{Y}_{\pm}$. These rays correspond to \newword{pure tables} with $\widetilde{\beta_{0,\lambda}} = \widetilde{\beta_{1,\mu}} = 1$ and all other entries zero. Up to scaling, these tables come from free resolutions of the form
\[M \leftarrow \mb{S}_{\lambda}(\mb{C}^k)^{\oplus c_0} \otimes R \leftarrow \mb{S}_{\mu}(\mb{C}^k)^{\oplus c_1} \otimes R \leftarrow 0,\]
with all generators in type $\lambda$ and all syzygies in type $\mu$.
\end{theorem}

We will need a slightly more general result for the purposes of the equivariant Boij-S\"{o}derberg pairing, a derived analog to $\rcone{k,n}$.

\begin{definition}
The \newword{derived Boij-S\"{o}derberg cone}, denoted $\ercone{k,n}$, is the positive linear span of (rank) Betti tables of bounded minimal complexes $F_\bullet$ of equivariant free modules, such that $F_\bullet$ is exact away from the locus of rank-deficient matrices.
\end{definition}
In this definition, we assume only that the homology modules $M$ have $\sqrt{\ann(M)} \supseteq P_k$, not that equality holds. We also do not assume Cohen-Macaulayness. Thus, $\ercone{k,k}$ includes, for example, homological shifts of elements of $\rcone{k,k}$, and Betti tables of longer complexes. The simplest tables in the derived cone are {\bf homologically shifted pure tables}, written $\shiftedpuretable{i}{\lambda}{\mu}$, for $i \in \mb{Z}$ and $\lambda \subsetneq \mu$. These are the tables with $\widetilde{\beta_{i,\lambda}} = \widetilde{\beta_{i+1,\mu}} = 1$ and all other entries zero. 

We show:
\begin{theorem}
The cone $\ercone{k,k}$ is rational polyhedral. Its extremal rays are the homological shifts of those of $\rcone{k,k}$, spanned by the tables $\shiftedpuretable{i}{\lambda}{\mu}$. The supporting hyperplanes are indexed by tuples $(\ldots, S_{-1}, S_1, S_3, \ldots)$ of convex subsets $S_i \subseteq \mb{Y}_{\pm}$, one chosen for every other spot along the complex.
\end{theorem}

The key idea in the above theorem is that these Betti tables are characterized by certain perfect matchings. This idea is also crucial in our construction of the pairing between Betti and cohomology tables, so we discuss it now. We introduce a graph-theoretic model of a rank Betti table (in the case of free resolutions, this construction is implicit in \cite[Lemma 3.6]{FLS16}).
\begin{definition}[Betti graphs] \label{def:betti-graph}
Let $\widetilde{\beta} \in \widetilde{\mb{BT}}_{k,k}$ have nonnegative integer entries. The \newword{Betti graph} $G(\widetilde{\beta})$ is defined as follows:
\begin{itemize}
\item The vertex set contains $\widetilde{\beta_{i,\lambda}}$ vertices labeled $(i,\lambda)$, for each $(i,\lambda)$,
\item The edge set contains, for each $i$, all possible edges $(i,\lambda) \leftarrow (i+1,\mu)$ with $\lambda \subsetneq \mu$.
\end{itemize}
Note that this graph is bipartite: every edge connects an even-indexed and an odd-indexed vertex.
\end{definition}
Recall that a {\bf perfect matching} on a graph $G$ is a subset of its edges, such that every vertex of $G$ appears on exactly one chosen edge. A perfect matching on $G(\widetilde{\beta})$ is equivalent to a decomposition of $\widetilde{\beta}$ as a positive integer combination of homologically-shifted pure tables: an edge $(i,\lambda) \leftarrow (i+1,\mu)$ corresponds to a pure summand $\shiftedpuretable{i}{\lambda}{\mu}$. Thus, an equivalent characterization of $\ercone{k,k}$ is:
\begin{theorem}
Let $\widetilde{\beta} \in \mb{BT}_{k,k}$ have nonnegative integer entries. Then $\widetilde{\beta} \in \ercone{k,k}$ if and only if $G(\widetilde{\beta})$ has a perfect matching.
\end{theorem}
Our proof proceeds by exhibiting this perfect matching using homological algebra. The supporting hyperplanes of $\ercone{k,k}$ then follow from Hall's Matching Theorem; see Section \ref{subsec:derived-cone} for the precise statement.

\subsubsection{The pairing between Betti tables and cohomology tables} \label{subsubsec:equi-pairing}

We now turn to the Boij-S\"{o}derberg pairing. This will be a bilinear pairing between abstract Betti tables $\beta$ and cohomology tables $\gamma$, satisfying certain nonnegativity properties when restricted to realizable tables.

\begin{definition} \label{def:pairing-intro} Let $\beta \in \mb{BT}_{k,n}$ and $\gamma \in \mb{CT}_{k,n}$ be an abstract Betti table and $GL$-cohomology table. The \newword{equivariant Boij-S\"{o}derberg pairing} is given by
\begin{equation}
\begin{split}
\widetilde{\Phi} : \mb{BT}_{k,n} \times \mb{CT}_{k,n} &\to \widetilde{\mb{BT}}_{k,k}, \\
(\beta,\gamma) \qquad &\mapsto \widetilde{\Phi}(\beta,\gamma),
\end{split}
\end{equation}
with $\widetilde{\Phi}$ the (derived) \emph{rank} Betti table with entries
\begin{equation} \label{eqn:pairing-entries}
\widetilde{\phi_{i,\lambda}}(\beta, \gamma) = \sum_{p-q=i} \beta_{p,\lambda} \cdot \gamma_{q,\lambda}.
\end{equation}
In this definition, recall that the homological index of a complex decreases under the boundary map.
\end{definition}

Here is how to read the definition of $\widetilde{\Phi}$. (See Example \ref{exa:pairing} below.) Form a grid in the first quadrant of the plane, whose $(p,q)$-entry is the collection of numbers $\beta_{p,\lambda} \cdot \gamma_{q,\lambda}$ for all $\lambda$. Only finitely-many of these are nonzero. The line $p-q = i$ is an upwards-sloping diagonal through this grid, and $\widetilde{\phi_{i,\lambda}}$ is the sum of the $\lambda$ terms along this diagonal.

\begin{remark} We emphasize that the pairing takes a \emph{multiplicity} Betti table $\beta$ and a cohomology table $\gamma$, and produces a \emph{rank} Betti table $\widetilde{\Phi}$. Intuitively, the entries of $\gamma$ are dimensions of certain vector spaces (from sheaf cohomology), which, we will see, arise with multiplicities given by $\beta$ in a certain spectral sequence. In particular, the quantities in \eqref{eqn:pairing-entries} are again dimensions of vector spaces -- that is, they give a rank table.
\end{remark}
The final main result of this paper is the nonnegativity of the pairing:
\begin{theorem}[Pairing the equivariant cones] \label{thm:pairing-intro}
The pairing $\widetilde{\Phi}$ restricts to a map of cones,
\[\mcone{k,n} \times \escone{k,n} \to \ercone{k,k}.\]
The same is true with $\mcone{k,n}$ replaced by $\emcone{k,n}$ on the source.
\end{theorem}
In particular, the defining inequalities of the cone $\ercone{k,k}$ (which we give explicitly) pull back to nonnegative bilinear pairings of Betti and cohomology tables, and the Betti graph of $\widetilde{\Phi}(\beta,\gamma)$ has a perfect matching. We think of this as a reduction to the base case of square matrices ($k=n$). A geometric consequence is that each equivariant Betti table induces many interesting linear inequalities constraining sheaf cohomology on $Gr(k,\mb{C}^n)$.

Our proof proceeds by constructing a perfect matching on $\widetilde{\Phi}(\beta,\gamma)$, not by passing to actual modules over $R_{k,k}$. It would be interesting to see a `categorified' form of the pairing, in the style of Eisenbud-Erman \cite{EE2012}. Such a pairing would construct, from a complex $F_\bullet$ of $R_{k,n}$-modules and a sheaf $\mc{E}$, a module (or complex) over $R_{k,k}$. Theorem \ref{thm:pairing-intro} would follow from showing that this module is supported along the determinant locus (or that the complex is exact away from the determinant locus).
%
The authors welcome any communication or ideas in this direction.

\begin{example} \label{exa:pairing}
Let us pair the following tables for $k = 2$, $n=3$. Both are realizable; the cohomology table is for the sheaf $\mc{E} = \mc{O}(1) \oplus \mc{O}(-1)$.
\begin{equation} \label{eqn:example-pairing}
\begin{array}{c|ccc}
\beta_{p,\lambda} & 0 & 1 & 2 \\ \hline
\scalebox{.5}{\yng(1)}   & 4 & - & - \\
\scalebox{.5}{\yng(2)}   & - & 1 & - \\
\scalebox{.5}{\yng(1,1)} & - & 9 & - \\
\scalebox{.5}{\yng(2,1)} & - & 3 & 3 \\
\scalebox{.5}{\yng(3,1)} & - & - & 1 \\
\scalebox{.5}{\yng(2,2)} & - & - & 1 \\
\end{array}
\qquad \times \qquad
\begin{array}{c|ccc}
\gamma_{q,\lambda} & 0 & 1 & 2 \\ \hline
\scalebox{.5}{\yng(1)}   & 3 & 1 & - \\
\scalebox{.5}{\yng(2)}   & - & 3 & - \\
\scalebox{.5}{\yng(1,1)} & 1 & - & - \\
\scalebox{.5}{\yng(2,1)} & - & - & - \\
\scalebox{.5}{\yng(3,1)} & - & 3 & - \\
\scalebox{.5}{\yng(2,2)} & - & - & 1 \\
\end{array}
\end{equation}
We arrange the pairwise products in a first-quadrant grid. The sums along the diagonals $\{p-q=i\}$ result in the rank Betti table $\widetilde{\Phi}$:
 \vspace{0.1cm}
\begin{equation}
\vcenter{\xymatrix{
- & - & 1 \cdot \scalebox{.5}{\yng(2,2)} \\
4 \cdot \scalebox{.5}{\yng(1)} & 3 \cdot \scalebox{.5}{\yng(2)} & 3 \cdot \scalebox{.5}{\yng(3,1)} \\
\ar@<6ex>@{..>}[]+/d 4.5ex/;[uu]+/u 3ex/^q \ar@<-5ex>@{..>}[]+/l 5ex/;[rr]+/r 5ex/_p
12 \cdot \scalebox{.5}{\yng(1)} & 9 \cdot \scalebox{.5}{\yng(1,1)}  & - }}
\qquad \leadsto \qquad
\begin{array}{c|ccc}
\widetilde{\phi_{i,\lambda}} & -1 & 0 & 1 \\ \hline
\scalebox{.5}{\yng(1)}   & 4 & 12 & - \\
\scalebox{.5}{\yng(2)}   & - & 3 & - \\
\scalebox{.5}{\yng(1,1)} & - & - & 9 \\
\scalebox{.5}{\yng(3,1)} & - & - & 3 \\
\scalebox{.5}{\yng(2,2)} & - & 1 & - \\
\end{array}
\end{equation}
Finally, we check that $\widetilde{\Phi} \in \ercone{k,k}$.  The decomposition of $\widetilde{\Phi}$ into pure tables happens to be unique (this is not true in general):
\[\widetilde{\Phi} =
3 \ \shiftedpuretable{-1}{\tiny \yng(1)}{\tiny \yng(2)}
+ \shiftedpuretable{-1}{\tiny \yng(1)}{\tiny \yng(2,2)}
+ 9 \ \shiftedpuretable{0}{\tiny \yng(1)}{\tiny \yng(1,1)}
+ 3 \ \shiftedpuretable{0}{\tiny \yng(1)}{\tiny \yng(3,1)}.\]
This corresponds to an essentially-unique perfect matching on $G(\widetilde{\Phi})$.
\end{example}

\subsection{Structure of the paper}
Section \ref{sec:background} contains background on algebra and representation theory. Sections \ref{sec:equi-hk}, \ref{sec:square-mats} and \ref{sec:the-pairing} respectively establish the equivariant Herzog-K\"{u}hl equations, the results on square matrices, and the pairing of Betti and cohomology tables. Finally, Section \ref{sec:2x3mats} has some preliminary results and examples for the case $k=2,n=3$.

\subsection{Acknowledgments}
\label{subsec:acknowledgments}

This work has benefited substantially from conversations and collaboration with many people: Daniel Erman (who introduced the second author to Boij-S\"{o}derberg theory), Steven Sam, David Speyer, Greg Muller and Maria Gillespie. We are particularly grateful to David Speyer for discussions surrounding Theorem \ref{thm:equi-hk} (the equivariant Herzog-K\"{u}hl equations), and to both Greg Muller and David Speyer for several productive discussions around the proof of the numerical pairing, Theorem  \ref{thm:pairing-intro}. 

Finally, computations of free resolutions and Betti tables in Macaulay2 \cite{M2}, and facets and rays of cones in Sage \cite{sage} (via SageMathCloud) have been, and continue to be, instrumental.

\section{Background} \label{sec:background}

\subsection{Spaces of interest}

Throughout, let $V, W$ be fixed $\mb{C}$-vector spaces of dimensions $k$ and $n$, with $k \leq n$. We set
\[X = \Hom(V,W), \qquad R_{k,n} = \Sym(\Hom(V,W)^*) \cong \mb{C}\Big[ x_{ij} : \begin{aligned} &1 \leq i \leq k \\ &1 \leq j \leq n\end{aligned} \Big],\]
so $X = \Spec(R_{k,n})$, the affine variety of $k \times n$ matrices, and $R_{k,n}$ is the polynomial ring whose variables are the entries of the matrix. We also consider the subvarieties of full-rank and rank-deficient matrices,
\[U = \Emb(V,W) = \{T : \ker(T) = 0\}, \qquad X_{k-1} = X - U,\]
which are open and closed, respectively. The locus $X_{k-1}$ is integral and has codimension $n-k+1$. Its prime ideal $P_k$ is generated by the $\binom{n}{k}$ maximal minors of the $k \times n$ matrix $(x_{ij})$. Each of the spaces $X$, $R_{k,n}$, $U$ and $X_{k-1}$ has an action of $GL(V)$ and $GL(W)$; we will primarily care about the $GL(V)$ action.

\subsection{\texorpdfstring{$GL$}{GL}-Representation theory}

A good introduction to these notions is \cite{Fulton}. The irreducible algebraic representations of $GL(V)$ are indexed by weakly-decreasing integer sequences $\lambda = (\lambda_1 \geq \cdots \geq \lambda_k)$, where $k = \dim(V)$. We write $\mb{S}_\lambda(V)$ for the corresponding representation, and $d_\lambda(k)$ for its dimension. We call $\mb{S}_\lambda$ a \newword{Schur functor}. If $\lambda$ has all nonnegative parts, we write $\lambda \geq 0$ and say $\lambda$ is a \newword{partition}. In this case, $\mb{S}_\lambda(V)$ is functorial for linear transformations $V \to W$. If $\lambda$ has negative parts, $\mb{S}_\lambda$ is only functorial for isomorphisms $V \xrightarrow{\sim} W$.

We often represent partitions by their Young diagrams:
\[
\lambda = (3,1) \longleftrightarrow \lambda = {\tiny \yng(3,1)}.
\]
We partially order partitions and integer sequences by containment:
\[
\lambda \subseteq \mu \text{ if } \lambda_i \leq \mu_i \text{ for all } i.
\]
We write $\mb{Y}$ for the poset of all partitions with this ordering, called \newword{Young's lattice}. We write $\mb{Y}_{\pm}$ for the set of all weakly-decreasing integer sequences; we call it the \newword{extended Young's lattice}. Schur functors include symmetric and exterior powers:
\begin{align*}
\lambda = d \bigg\{ {\tiny \young(\hfil,\hfil,\hfil,\hfil)}\hspace{0.5cm} &\Longleftrightarrow\ \mb{S}_\lambda(V) = \Alt^d(V),\\
\lambda =  \overbrace{{\tiny \young(\hfil\hfil\hfil\hfil)}}^d\ &\Longleftrightarrow\ \mb{S}_\lambda(V) = \Sym^d(V).
\end{align*}
We'll write $\det(V)$ for the one-dimensional representation $\Alt^{\dim(V)}(V)=\mb{S}_{1^k}(V)$. We may always twist a representation by powers of the determinant:
\[
\det(V)^{\otimes a} \otimes \mb{S}_{\lambda_1, \ldots, \lambda_k}(V) = \mb{S}_{\lambda_1 + a, \ldots, \lambda_k + a}(V) 
\]
for any integer $a \in \mb{Z}$. This operation is invertible and can sometimes be used to reduce to considering the case when $\lambda$ is a partition.

\subsection{Equivariant rings and modules}

If $R$ is a $\mb{C}$-algebra with an action of $GL(V)$, and $S$ is any $GL(V)$-representation, then $S \otimes_\mb{C} R$ is an \newword{equivariant free $R$-module}; it has the universal property
\[\Hom_{GL(V),R}(S \otimes_\mb{C} R, M) \cong \Hom_{GL(V)}(S,M)\]
for all equivariant $R$-modules $M$. The basic examples will be the modules $\mb{S}_\lambda(V) \otimes R$.

Let $R = R_{k,n}$ be the polynomial ring defined above. Its structure as a $GL(V) \times GL(W)$ representation is known as the \newword{Cauchy identity}:
\[R_{k,n} = \Sym^\bullet(\Hom(V,W)^*) \cong \bigoplus_{\lambda \geq 0} \mb{S}_\lambda(V) \otimes \mb{S}_\lambda(W^*).\]
Note that the prime ideal $P_k$ and the maximal ideal $\mf{m} = (x_{ij})$ of the zero matrix are $GL(V)$- and $GL(W)$-equivariant.

Let $M$ be a finitely-generated $GL(V)$-equivariant $R$-module. The module $\Tor_R^i(R/\mf{m},M)$ naturally has the structure of a finite-dimensional $GL(V)$-representation. We define the \newword{equivariant Betti number} $\beta_{i,\lambda}(M)$ as the multiplicity of the Schur functor $\mb{S}_\lambda(V)$ in this Tor module, i.e.
\begin{equation*}\label{eqn:multiplicity-betti-number}
\Tor_R^i(R/\mf{m},M)\ \cong\ \bigoplus_\lambda \mb{S}_\lambda(V)^{\oplus \beta_{i,\lambda}(M)} \qquad \text{(as $GL(V)$-representations).}
\end{equation*}
By semisimplicity of $GL(V)$-representations, any minimal free resolution of $M$ can be made equivariant, so we may instead define $\beta_{i,\lambda}$ as the multiplicity of the equivariant free module $\mb{S}_\lambda(V) \otimes R$ in the $i$-th step of an equivariant minimal free resolution of $M$:
\begin{equation*}
M \leftarrow F_0 \leftarrow F_1 \leftarrow \cdots \leftarrow F_d \leftarrow 0, \text{ where } F_i = \bigoplus_\lambda \mb{S}_\lambda(V)^{\beta_{i,\lambda}(M)} \otimes R.
\end{equation*}
All other notation on Betti tables is as defined in Section \ref{subsec:grass-boijsoderberg}.

\section{The equivariant Herzog-K\"{u}hl equations}
\label{sec:equi-hk}

In this section we derive the equivariant analogue of the Herzog-K\"{u}hl equations. This will be a system of linear conditions on the entries of an equivariant Betti table. It will detect when the resolved module $M$ is supported only along the locus of rank-deficient matrices.

\subsection{K-theory rings} \label{subsec:ktheory-bg}

For background on equivariant $K$-theory, we refer to the original paper by Thomason \cite{Th87}; a more recent discussion is \cite{Me05}.

Excision in equivariant K-theory (\cite[Theorem 2.7]{Th87}) gives the right-exact sequence of abelian groups
\[K^{GL(V)}(X_{k-1}) \xrightarrow{i_*} K^{GL(V)}(\Hom(V,W)) \xrightarrow{j^*} K^{GL(V)}(U) \to 0.\]
The pullback $j^*$, induced by the open inclusion $j : U \into X$, is a map of rings. The pushforward $i_*$, induced by the closed embedding $i : X_{k-1} \into X$, is only a map of abelian groups. Its image is the ideal $I$ generated by the classes of modules supported along the rank-deficient locus $X_{k-1}$.

We do not attempt to describe the first term. For the second term, we have (\cite[Theorem 4.1]{Th87} or \cite[Example 2 and Corollary 12]{Me05})
\[K^{GL(V)}(\Hom(V,W)) \cong \mb{Z}[x_1^\pm, \ldots, x_k^\pm]^{S_k},\]
the ring of symmetric Laurent polynomials in $k$ variables (essentially the representation ring of $GL(V)$). Here, the class of the equivariant $R$-module $\mb{S}_\lambda(V) \otimes_\mb{C} R$ is identified with the Schur polynomial $s_\lambda(t_1, \ldots, t_k)$. If $M$ is a finitely-generated $R$-module, its equivariant minimal free resolution expresses the K-class $[M]$ as a finite alternating sum of Schur polynomials. In other words, the equivariant Betti table determines the K-class:
\[[M] = \sum_{i,\lambda} (-1)^i \beta_{i,\lambda}(M) s_\lambda(t).\]
An equivalent approach is to write
\[M \cong \bigoplus_\lambda \mb{S}_\lambda(V)^{c_\lambda(M)} \text{ as a $GL(V)$-representation},\]
and define the equivariant Hilbert series of $M$,
\begin{align*}
H_M(t) &= \sum_\lambda c_\lambda(M) s_\lambda(t) \\
&=\frac{f(t)}{\prod_{i=1}^k(1-t_i)^n}
\end{align*}
for some symmetric function $f(t)$. Then $f(t)$ is the K-theory class of $M$. (If we forget the $GL(V)$ action and remember only the grading of $M$, we recover the usual Hilbert series.)

To see that these definitions agree, note that the second definition is additive in short exact sequences, hence is well-defined on K-classes. Replacing $M$ by its equivariant minimal free resolution, it suffices to consider indecomposable free modules $M = \mb{S}_\lambda(V) \otimes_\mb{C} R$. Then the Cauchy identity shows $H_M(t) = s_\lambda(t)$.

It will be convenient in this section to restrict to modules $M$ generated in positive degree, i.e. $\beta_{i,\lambda}(M) \ne 0$ implies $\lambda \geq 0$. In this case, the class of $M$ is a polynomial, not a Laurent polynomial. We write $K^{GL(V)}_+(\Hom(V,W))$ for this subring.

Finally, we have for the third term (cf. \cite[Proposition 6.2]{Th87}) \[K^{GL(V)}(U) \cong K(U/GL(V)) = K(Gr(k,W)),\]
because the action of $GL(V)$ is free on $U$. The structure of this ring is well-known from K-theoretic Schubert calculus (e.g. \cite{KK90} or \cite{Bu02}). We will only need to know the following: it is a free abelian group with an additive basis consisting of $\binom{n}{k}$ generators, indexed by partitions $\mu$ fitting inside a $k \times (n-k)$ rectangle. These correspond to the classes $[\mc{O}_\mu]$ of structure sheaves of Schubert varieties. It is easy to check that $K^{GL(V)}_+(\Hom(V,W)) \to K_0(Gr(k,W))$ is also surjective (because, e.g., matrix Schubert varieties are generated in positive degree).

\subsection{Modules on the rank-deficient locus and the equivariant Herzog-K\"uhl equations} \label{subsec:FLSHK}

From the surjection $K^{GL(V)}_+(\Hom(V,W)) \to K_0(Gr(k,W))$, we see that the ideal 
\[I' := I \cap K^{GL(V)}_+(\Hom(V,W)),\] as a linear subspace, has co-rank $\binom{n}{k}$. We wish to find exactly this many linear equations cutting out the ideal, indexed appropriately by partitions. That is, given a K-class written in the Schur basis,
\[f = \sum_{\lambda \geq 0} a_\lambda s_\lambda \in K^{GL(V)}_+(\Hom(V,W)),\]
we wish to have coefficients $b_{\lambda \mu}$ for each $\mu \subseteq \rect$, such that
\[f \in I' \text{ if and only if } \sum_{\lambda \geq 0} a_\lambda b_{\lambda \mu} = 0 \text{ for all } \mu \subseteq \rect\ .\]
We will then apply these equations in the case where $f$ is the class of a module $M$, and \[a_\lambda = \sum_i (-1)^i \beta_{i,\lambda}(M)\] comes from the equivariant Betti table of $M$. Our approach is to prove the following:
\begin{theorem} \label{thm:basis-for-I}
We have $I' = \mr{span}_\mb{C} \big\{ s_\lambda(1-t) : \lambda \not\subseteq \rect\big\}$.
\end{theorem}
\noindent We will prove Theorem \ref{thm:basis-for-I} in the next section. Here is how it leads to the desired equations. Let $b_{\lambda \mu}$ be the change-of-basis coefficients defined by sending $t_i \mapsto 1-t_i$. So, by definition,
\[s_\lambda(1-t) = \sum_\mu b_{\lambda \mu} s_\mu(t).\]
Note that we have, equivalently,
\[s_\lambda(t) = \sum_\mu b_{\lambda \mu} s_\mu(1-t).\]
Thus
\[f = \sum_\lambda a_\lambda s_\lambda(t) = \sum_{\lambda,\mu} a_\lambda b_{\lambda \mu} s_\mu(1-t).\]
The polynomials $s_\mu(1-t)$ for all $\mu\geq 0$ form an additive basis for the $K^{GL(V)}_+(\Hom(V,W))$. Thus, $f \in I'$ if and only if the coefficient of $s_\mu(1-t)$ is $0$ for all $\mu \subseteq \rect$. That is,
\[0 = \sum_\lambda a_\lambda b_{\lambda \mu} \text{ for all } \mu \subseteq \rect\ .\]
The following description of $b_{\lambda \mu}$ is due to Stanley. Recall that, if $\mu \subseteq \lambda$ are partitions, the \newword{skew shape} $\lambda/\mu$ is the Young diagram of $\lambda$ with the squares of $\mu$ deleted. A \newword{standard Young tableau} is a filling of a (possibly skew) shape by the numbers $1, 2, \ldots, t$ (with $t$ boxes in all), such that the rows increase from left to right, and the columns increase from top to bottom. We write $f^\sigma$ for the number of standard Young tableaux of shape $\sigma$.
\begin{proposition}\cite{StanEC2} \label{prop:stanley-change-basis}
If $\mu \not\subseteq \lambda$ then $b_{\lambda \mu} = 0$. If $\mu \subseteq \lambda$, then
\[b_{\lambda/ \mu} = (-1)^{|\mu|} \frac{f^{\lambda/ \mu} f^\mu}{f^\lambda} \binom{|\lambda|}{|\mu|} \frac{d_\lambda(k)}{d_\mu(k)}.\]
An equivalent formulation is
\[b_{\lambda \mu} = (-1)^{|\mu|} \frac{f^{\lambda/\mu}}{|\lambda/\mu|!} \prod_{(i,j) \in \lambda/\mu}(k+j-i).\]
\end{proposition}

\begin{corollary}[Equivariant Herzog-K\"{u}hl equations] \label{cor:FLSHK}
Let $M$ be an equivariant $R$-module with equivariant Betti table $\beta_{i,\lambda}$. Assume $M$ is generated in positive degree.

The set-theoretic support of $M$ is contained in the rank-deficient locus if and only if:
\begin{equation} \label{eqn:equi-hk}
\text{ For each } \mu \subseteq \rect: \ \ \sum_{i,\lambda \supseteq \mu} (-1)^i \underbrace{\beta_{i,\lambda} d_\lambda(k)}_{(= \widetilde{\beta_{i,\lambda}})}  \frac{f^{\lambda/ \mu} f^\mu}{f^\lambda} \binom{|\lambda|}{|\mu|} = 0.
\end{equation}
Note that $\beta_{i,\lambda}$ is the {\bf multiplicity} of the $\lambda$-isotypic component of the resolution of $M$ (in cohomological degree $i$), whereas $\beta_{i,\lambda} d_\lambda(k) = \widetilde{\beta_{i,\lambda}}$ is the {\bf rank} of this isotypic component.
\end{corollary}

\begin{proof}
($\Rightarrow$): The only thing to note is that, for simplicity, we have rescaled the $\mu$-indexed equation by $(-1)^{|\mu|} d_\mu(k)$.

($\Leftarrow$): If the equations are satisfied, then $M$ maps to the trivial K-theory class on $Gr(k,\mb{C}^n)$. Since the Grassmannian is projective, it follows that the sheaf associated to $M$ is zero. This implies the support restriction.
\end{proof}

The coefficient in Equation \eqref{eqn:equi-hk} has the following interpretation. Consider a uniformly-random filling $T$ of the shape $\lambda$ by the numbers $1, \ldots, |\lambda|$. Say that $T$ \newword{splits along} $\mu \sqcup \lambda/\mu$ if the numbers $1, \ldots, |\mu|$ lie in the subshape $\mu$. Then:
\begin{equation} \label{eqn:prob-interpretation}
\frac{f^{\lambda/\mu} f^\mu}{f^\lambda} \binom{|\lambda|}{|\mu|} = \frac{\mr{Prob}(T \text{ splits along } \mu \sqcup \lambda/\mu\  |\ T \text{ is standard})}{\mr{Prob}(T \text{ splits along } \mu \sqcup \lambda/\mu)}.
\end{equation}

\subsection{Proof of Theorem \ref{thm:basis-for-I}}

First, we recall the following fact about K-theory of Grassmannians:

\begin{proposition}[\cite{fink2012}, page 21]
The following identity holds of formal power series over $K(Gr(k,W))$:
\[\bigg(\sum_p \big[\bigwedge{}^p \mc{S}\big] u^p \bigg) \cdot \bigg(\sum_q \big[\bigwedge{}^q \mc{Q}\big] u^q \bigg)= (1+u)^n. \]
It is essentially a consequence of the tautological exact sequence of vector bundles
\[0 \to \mc{S} \to W \to \mc{Q} \to 0.\]
\end{proposition}
\noindent We rearrange this as
\begin{align*}
\bigg(\sum_q \big[\bigwedge{}^q \mc{Q}\big] u^q \bigg)
&= (1+u)^n \cdot \frac{1}{\bigg(\sum_p \big[\bigwedge{}^p \mc{S}\big] u^p \bigg)} \\
&= (1+u)^n \cdot \sum_p \big (-1)^p [\Sym^p(\mc{S})\big] u^p.
\end{align*}
The key observation is that the left-hand side is a polynomial in $u$ of degree $n-k$, since $\mc{Q}$ has rank $n-k$. Thus, the coefficient $f_\ell$ of $u^\ell$ of the right-hand side vanishes for $\ell \geq n-k$. In other words, viewing $f_\ell$ as a symmetric polynomial, we have $f_\ell \in I'$ for $\ell > n-k$.

We compute the coefficient $f_\ell$. Recall that $[\Sym^p(\mc{S})] = h_p$, the $p$-th homogeneous symmetric polynomial. We have
\begin{align*}
\sum_\ell f_\ell u^\ell &= (1+u)^n \sum_p (-1)^p h_p u^p \\
&= \sum_{q=0}^n \sum_{p = 0}^\infty u^{p+q} (-1)^p h_p \binom{n}{q} \\
&= \sum_{\ell=0}^\infty u^\ell \sum_{p=\ell-n}^\ell (-1)^p h_p \binom{n}{\ell-p}, \\
\intertext{so our desired coefficients are}
f_\ell &= \sum_{p=\ell - n}^\ell (-1)^p h_p \binom{n}{\ell-p},
\end{align*}
where in the last two lines we use the convention $h_p = 0$ for $p < 0$. We next show:

\begin{lemma}
We have $I' \supseteq (h_{n-k+1}(1-t), \ldots, h_n(1-t))$.
\end{lemma}

\begin{proof}
Equivalently, we change basis $t \mapsto 1-t$, calling the (new) ideal $J$, and we show
\[J \supseteq (h_{n-k+1}, \ldots, h_n).\]
We consider the elements $f_{n-k+i}(1-t) \in J$ for $i = 1, \ldots, k$.
\begin{align*}
f_{n-k+i}(1-t) &= \sum_{p=-k+i}^{n-k+i} (-1)^p h_p(1-t) \binom{n}{n-k+i-p}.
\intertext{Since $i \leq k$ we have}
&= \sum_{p=0}^{n-k+i} (-1)^p h_p(1-t) \binom{n}{n-k+i-p}.
\end{align*}
We apply the second formula from Proposition \ref{prop:stanley-change-basis}. Note that all terms are single-row partitions, $\lambda = (p)$ and $\mu = (s)$, with $s\leq p$, so $f^{\lambda/\mu} = 1$ and the change of basis is:
\[b_{\lambda \mu} = (-1)^{s} \frac{f^{\lambda/\mu}}{|\lambda/\mu|!} \cdot (k+s) \cdots (k+p-1) = (-1)^s \binom{k+p-1}{k+s-1}.\]
Hence,
\begin{align*}
f_{n-k+i}(1-t) &= \sum_{p=0}^{n-k+i} \sum_{s=0}^p (-1)^{p+s} \binom{n}{n-k+i-p} \binom{k+p-1}{k+s-1} h_s. \\
&= \sum_{s=0}^{n-k+i} (-1)^s h_s \sum_{p=s}^{n-k+i} (-1)^p \binom{n}{n-k+i-p} \binom{k+p-1}{k+s-1}.
\intertext{We reindex, sending $p \mapsto n-k+i-p$, and reverse the order of the inner sum:}
&= (-1)^{n-k+i} \sum_{s=0}^{n-k+i} (-1)^s h_s \sum_{p=0}^{n-k+i-s} (-1)^p \binom{n}{p} \binom{n+i- p-1}{k+s-1}.
\end{align*}
{\bf The terms $h_s$ for $s \leq n-k$.}
First, we show that all the lower terms $h_s$, with $s \leq n-k$, vanish. For these terms, we view the large binomial coefficient as a polynomial function of $p$. It has degree $k+s-1$, with zeroes at $p = (n-k+i-s)+1, \ldots, n+i-1$, so we may freely include these terms in the inner sum. It is convenient to extend the inner sum only as far as $p=n$, obtaining
\[\sum_{p=0}^{n} (-1)^p \binom{n}{p} \binom{n-i-p-1}{k+s-1}.\]
Recall from the theory of finite differences that
\[\sum_{p=0}^d (-1)^p \binom{d}{p} g(p) = 0\] whenever $g$ is a polynomial of degree $< d$. Since the above sum has degree $k+s-1 \leq n-1$, it vanishes. Thus, dropping the lower terms, we are left with
\begin{align} \label{eqn:leftover-finite-diff}
f_{n-k+i}(1-t) &= (-1)^i \sum_{s=1}^i (-1)^s h_{n-k+s} \sum_{p=0}^{i-s} (-1)^p \binom{n}{p} \binom{n+i- p-1}{n+s-1}.
\end{align}
{\bf Showing $h_{n-k+i} \in J$ for $i=1, \ldots, k$.}
From equation \eqref{eqn:leftover-finite-diff}, we see directly that the coefficient of $h_{n-k+i}$ in $f_{n-k+i}(1-t)$ is 1. This is the leading coefficient, so the claim follows by induction on $i$.
\end{proof}

\begin{corollary}
We have
\[J = (h_i : i > n-k) = \mathrm{span}_\mb{C} \big\{s_\lambda : \lambda \not\subseteq \rect \big\}.\]
\end{corollary}
\begin{proof}
The equality of ideals
\[(h_{n-k+1}, \ldots, h_n) = (h_i : i > n-k)\]
follows from Newton's identities and induction. The equality
\[(h_i : i > n-k) = \mathrm{span}_\mb{C} \big\{ s_\lambda : \lambda \not\subseteq \rect\big\}\]
follows from the Pieri rule (for $\subseteq$) and the Jacobi-Trudi formula (for $\supseteq$). See \cite{Fulton} for these identities. This shows $J$ \emph{contains} this linear span. But then quotienting by $J$ leaves at most $\binom{n}{k}$ classes. This is already the rank of $K(Gr(k,W))$, so we must have equality.
\end{proof}

Changing bases $t \mapsto 1-t$ a final time completes the proof of Theorem \ref{thm:basis-for-I}.

\section{Square matrices and perfect matchings}
\label{sec:square-mats}

\begin{remark}
In this section, rank Betti tables play a more significant role than multiplicity tables. As such, we will state results in terms of the cones $\widetilde{BS}_{k,k}$ and $\ercone{k,k}$.
\end{remark}

We now describe the Boij-S\"{o}derberg cone in the base case of square matrices; thus we set $n=k$ for the remainder of this section. The corresponding Grassmannian is a point, so there is no dual geometric picture or cone. We will recall the description of $\rcone{k,k}$ due to \cite{FLS16}; we then describe the derived cone $\ercone{k,k}$.

When $k=1$, the ring is just $\mb{C}[t]$, and its torsion graded modules are essentially trivial to describe. See Section 4 of \cite{EE2012} for a short, complete description of both cones. For $k > 1$, however, the cones are algebraically and combinatorially interesting, although simpler than the general case.

\subsection{Prior work on \texorpdfstring{$\rcone{k,k}$ (\cite{FLS16})}{rank Betti tables for square matrices}} \label{subsec:squaremat-basic-cone}

The rank-deficient locus $\{\det(T) = 0\} \subset \Hom(V,W)$ is codimension 1. Thus, modules satisfying Condition \ref{cond:modules-of-interest} have free resolutions of length 1,
\[M \leftarrow F^0 \leftarrow F^1 \leftarrow 0.\]
There is only one equivariant Herzog-K\"{u}hl equation, labeled by the empty partition $\mu = \eset$:
\begin{equation} \label{eqn:rank-equation}
\sum_\lambda \widetilde{\beta_{0,\lambda}} = \sum_\lambda \widetilde{\beta_{1,\lambda}},\ \text{ that is, }\ \rank(F^0) = \rank(F^1).
\end{equation}
Algebraically, this simply says that $M$ is a torsion module.\\

The extremal rays and supporting hyperplanes of $\rcone{k,k}$ are as follows.
\begin{definition}[Pure tables] \label{def:pure-tables}
Fix $\lambda, \mu \in \mb{Y}_{\pm}$ with $\lambda \subsetneq \mu$. The \newword{pure table} $\widetilde{\beta}(\lambda \subsetneq \mu)$ is defined by setting
\[\widetilde{\beta_{0,\lambda}} = \widetilde{\beta_{1,\mu}} = 1\]
and all other entries 0.
\end{definition}
It is nontrivial to show that each pure table $\widetilde{\beta}(\lambda \subsetneq \mu)$ is realizable up to scalar multiple (\cite[Theorem 4.1]{FLS16}). 
Any such table generates an extremal ray of $\rcone{k,k}$.

It is, by contrast, easy to establish the following inequalities on $\rcone{k,k}$.
\begin{definition}[Antichain inequalities] \label{def:antichain-ineqs}
Let $S \subseteq \mb{Y}_{\pm}$ be a downwards-closed set. Let
\[\Gamma = \{\lambda : \lambda \subsetneq \mu \text{ for some } \mu \in S\}.\]
For any rank Betti table $(\widetilde{\beta_{i,\lambda}})$, the \newword{antichain inequality (for $S$)}  is then:
\begin{equation} \label{ineq:antichains}
\sum_{\lambda \in \Gamma} \widetilde{\beta_{0,\lambda}} \geq \sum_{\lambda \in S} \widetilde{\beta_{1,\lambda}}.
\end{equation}
(The terminology of `antichains' is due to \cite{FLS16}, where the inequality \eqref{ineq:antichains} is stated in terms of the maximal elements of $S$, which form an antichain in $\eylat$.)
\end{definition}
\noindent These conditions follow directly from minimality of the underlying maps of modules: the summands corresponding to $S$ in $F_1$ must map into those corresponding to $\Gamma$ in $F_0$.

Finally, we recall the graph-theoretic model of $\widetilde{\beta}$ introduced in Section \ref{subsubsec:base-case}. This construction was implicit in \cite[Lemma 3.6]{FLS16}. It is especially simple in this case:
\begin{definition} \label{def:bipartite-betti-graph}
The \newword{Betti graph} $G(\widetilde{\beta})$ is the directed bipartite graph with left vertices $L$ and right vertices $R$, defined as follows:
\begin{itemize}
\item The set $L$ (resp. $R$) contains $\widetilde{\beta_{0,\lambda}}$ (resp. $\widetilde{\beta_{1,\lambda}}$) vertices labeled $\lambda$, for each $\lambda$,
\item The edge set contains all possible edges $\lambda \leftarrow \mu$, from $R$ to $L$, for $\lambda \subsetneq \mu$.
\end{itemize}
\end{definition}

The Boij-S\"{o}derberg cone $\rcone{k,k}$ is characterized as follows.

\begin{theorem}[{\cite[Theorem 3.8]{FLS16}}] \label{thm:square-matrices-cone}
The cone $\rcone{k,k}$ is defined by the rank equation \eqref{eqn:rank-equation}, the conditions $\widetilde{\beta_{i,\lambda}} \geq 0$, and the antichain inequalities \eqref{ineq:antichains}. Its extremal rays are the pure tables $\widetilde{\beta}(\lambda \subsetneq \mu)$, for all choices of $\lambda \subsetneq \mu$ in $\mb{Y}_{\pm}$. 

Moreover, if $\widetilde{\beta} \in \widetilde{\mb{B}}_{k,k}$ has nonnegative integer entries, then $\widetilde{\beta} \in \rcone{k,k}$ if and only if the Betti graph $G(\widetilde{\beta})$ has a perfect matching.
\end{theorem}
A perfect matching on $G(\widetilde{\beta})$ expresses $\widetilde{\beta}$ as a positive integer sum of pure tables: an edge $\lambda \leftarrow \mu$ corresponds to a summand 
\[\widetilde{\beta} = \cdots + \widetilde{\beta}(\lambda \subsetneq \mu) + \cdots.\]
It is easy to see that the cone spanned by the pure tables is contained in the cone defined by the antichain inequalities. The fact that these cones agree follows from Hall's Matching Theorem for bipartite graphs:

\begin{theorem}[Hall's Matching Theorem]
Let $G$ be a bipartite graph with left vertices $L$ and right vertices $R$, with $|L| = |R|$. Then $G$ has a perfect matching if and only if the following holds for all subsets $S \subseteq R$ (equivalently, for all subsets $S \subseteq L$): let $\Gamma(S)$ be the set of vertices adjacent to $S$. Then $|\Gamma(S)| \geq |S|$.
\end{theorem}
In the antichain inequality \eqref{ineq:antichains}, $S$ corresponds to a set of vertex labels on the right-hand-side of the Betti graph $G(\widetilde{\beta})$. The set $\Gamma$ consists of the labels of vertices adjacent to $S$. The numbers of such vertices are then the right- and left-hand-sides of the inequality. (The structure of $G(\widetilde{\beta})$ implies easily that it suffices to consider inequalities from downwards-closed sets $S$.)

\subsection{The derived cone}
\label{subsec:derived-cone}

We now generalize Theorem \ref{thm:square-matrices-cone} to describe the derived cone $\ercone{k,k}$. We are interested in bounded free equivariant complexes
\[\cdots \leftarrow F_i \leftarrow F_{i+1} \leftarrow F_{i+2} \leftarrow \cdots,\]
all of whose homology modules are torsion.

The supporting hyperplanes of $\ercone{k,k}$ are quite complicated and we do not establish them directly. We instead generalize the descriptions in terms of extremal rays and perfect matchings, which remain fairly simple. We then deduce the inequalities from Hall's Theorem.

The extremal rays of $\ercone{k,k}$ will be homological shifts of those of $\rcone{k,k}$:
\begin{definition}[Homologically-shifted pure tables]\label{def:hom-shifted-pure-tables}
Fix $i \in \mb{Z}$ and $\lambda, \mu \in \mb{Y}_{\pm}$ with $\lambda \subsetneq \mu$. We define the \newword{homologically-shifted pure table} $\shiftedpuretable{i}{\lambda}{\mu}$ by setting
\[\widetilde{\beta_{i,\lambda}} = \widetilde{\beta_{i+1,\mu}} = 1\]
and all other entries 0.
\end{definition}

\noindent The supporting hyperplanes will be defined by the following inequalities. Recall that a \newword{convex subset} $S$ of a poset $P$ is the intersection of an upwards-closed set with a downwards-closed set.
\begin{definition}[Convexity inequalities] \label{def:segment-ineqs}
For each odd $i$, let $S_i \subseteq \mb{Y}_{\pm}$ be any convex set. For each even $i$, define
\[\Gamma_i = \{\lambda : \mu \subsetneq \lambda \text{ for some } \mu \in S_{i-1}\} \cup \{\lambda : \lambda \subsetneq \mu \text{ for some } \mu \in S_{i+1}\}.\]
For any rank Betti table $(\widetilde{\beta_{i,\lambda}})$, the \newword{convexity inequality (for the $S_i$'s)}  is then:
\begin{equation} \label{ineq:segments}
\sum_{i \text{ even}} \sum_{\lambda \in \Gamma_i} \widetilde{\beta_{i,\lambda}} \geq \sum_{i \text{ odd}} \sum_{\lambda \in S_i} \widetilde{\beta_{i,\lambda}}.
\end{equation}
(We may, if we wish, switch `even' and `odd' in this definition. We will see that either collection of inequalities yields the same cone.)
\end{definition}

We recall the general definition of the Betti graph:
\begin{definition}[Betti graphs for complexes]
Let $\widetilde{\beta} \in \widetilde{\mb{B}}_{k,k}$ have nonnegative integer entries. The \newword{Betti graph} $G(\widetilde{\beta})$ is defined as follows:
\begin{itemize}
\item The vertex set contains $\widetilde{\beta_{i,\lambda}}$ vertices labeled $(i,\lambda)$, for each $(i,\lambda)$,
\item The edge set contains, for each $i$, all possible edges $(i,\lambda) \leftarrow (i+1,\mu)$ with $\lambda \subsetneq \mu$.
\end{itemize}
Note that this graph is bipartite: every edge connects an even-indexed and an odd-indexed vertex.
\end{definition}

Each segment $S_i$ of Definition \ref{def:segment-ineqs} corresponds to a set of vertex labels in $G(\widetilde{\beta})$. The set $\Gamma_i$ then contains the labels of vertices adjacent to $S_{i-1}$ and $S_{i+1}$. The numbers of vertices counted this way give the right- and left-hand-sides of the inequality \eqref{ineq:segments}. Note that if the $S_i$'s were not convex, we could replace them by their convex hulls without changing the $\Gamma_i$'s. \\

We now characterize the derived Boij-S\"{o}derberg cone $\ercone{k,k}$.

\begin{theorem}[The derived Boij-S\"{o}derberg cone, for square matrices] \label{thm:derived-cone}
Let $\widetilde{\beta}$ be an abstract rank Betti table. Without loss of generality, assume the entries of $\widetilde{\beta}$ are nonnegative integers. The following are equivalent:
\begin{itemize}
\item[(i)] $\widetilde{\beta} \in \ercone{k,k}$;
\item[(ii)] $\widetilde{\beta}$ satisfies all the convexity inequalities, together with the rank condition 
\[\sum_{i,\lambda} (-1)^i \widetilde{\beta_{i,\lambda}} = 0;\]
\item[(iii)] $\widetilde{\beta}$ is a positive integral linear combination of homologically shifted pure tables;
\item[(iv)] The Betti graph $G(\widetilde{\beta})$ has a perfect matching.
\end{itemize}
\end{theorem}
\begin{remark}
It is clear that (iv) $\Rightarrow$ (iii): each edge of a perfect matching indicates a pure table summand for $\widetilde{\beta}$. We have (iii) $\Rightarrow$ (ii) since the conditions (ii) hold for each homologically-shifted pure table individually. Hall's Matching Theorem gives the statement (ii) $\Leftrightarrow$ (iv) and shows that we may exchange `even' and `odd' in the definition of the convexity inequalities. Homologically-shifted pure tables are realizable, hence (iii) $\Rightarrow$ (i). We will complete the proof by exhibiting a perfect matching on any realizable Betti graph, so that (i) $\Rightarrow$ (iv). We postpone the proof until Section \ref{subsubsec:les-matching} (Corollary \ref{cor:derived-cone-proof}).
\end{remark}

\section{The pairing between Betti tables and cohomology tables}
\label{sec:the-pairing}

In this section, we establish the numerical pairing between Betti tables and cohomology tables. We recall that the pairing is defined as follows (Definition \ref{def:pairing-intro}):
\begin{equation}
\begin{split}
\widetilde{\Phi} : \mb{BT}_{k,n} \times \mb{CT}_{k,n} &\to \widetilde{\mb{BT}}_{k,k}, \\
(\beta,\gamma) \qquad &\mapsto \widetilde{\Phi}(\beta,\gamma),
\end{split}
\end{equation}
with $\widetilde{\Phi}$ the (derived) rank Betti table with entries
\begin{equation} \label{eqn:pairing-entries-recall}
\widetilde{\phi_{i,\lambda}}(\beta, \gamma) = \sum_{p-q=i} \beta_{p,\lambda} \cdot \gamma_{q,\lambda}.
\end{equation}
Recall also that the convention is that homological degree ($p$ and $i$) decreases under the boundary map of the complex.

The remainder of this section is devoted to the proof of the following:

\begin{theorem}[Pairing the equivariant Boij-S\"{o}derberg cones] \label{thm:pairing}
The pairing $\widetilde{\Phi}$ restricts to a pairing of cones,
\[\emcone{k,n} \times \escone{k,n} \to \ercone{k,k}.\]
\end{theorem}
In light of our description (Theorem \ref{thm:derived-cone}) of the derived cone $\ercone{k,k}$, the goal will be to exhibit a perfect matching on the Betti graph of $\widetilde{\Phi}(\beta,\gamma)$. Along the way, we will also complete the proof of Theorem \ref{thm:derived-cone} itself, showing that $\ercone{k,k}$ is characterized by the existence of such matchings (Corollary \ref{cor:derived-cone-proof}). \\

We sketch the construction of the pairing.

\begin{proof}[Sketch of Theorem \ref{thm:pairing}]
Let $\beta = \beta(F^\bullet)$ be the Betti table of a minimal free equivariant complex $F^\bullet$ of finitely-generated $R$-modules, with $R = R_{k,n}$ the coordinate ring of the $k \times n$ matrices. Assume $F^\bullet$ is exact away from the locus of rank-deficient matrices, so descending $F^\bullet$ to $Gr(k,\mb{C}^n)$ gives an exact sequence of vector bundles $\ms{F}^\bullet$:
\begin{equation} \label{eqn:fbullet-sum-decomp}
F^\bullet = \bigoplus_\lambda \mb{S}_\lambda(V)^{\beta_{\bullet,\lambda}} \otimes R \qquad \xrightarrow{\text{descends to}} \qquad \ms{F}^\bullet = \bigoplus_\lambda \mb{S}_\lambda(\mc{S})^{\beta_{\bullet,\lambda}},
\end{equation}
with $\mc{S}$ the tautological subbundle on $Gr(k,\mb{C}^n)$. Let $\gamma = \gamma(\mc{E})$ be the $GL$-cohomology table of a coherent sheaf $\mc{E}$ on $Gr(k,n)$. Observe that $\mc{E} \otimes \ms{F}^\bullet$ is again exact.

We study the hypercohomology spectral sequence. Explicitly, we take the Cech resolution of $\mathcal{E} \otimes \mathscr{F}^\bullet$, an exact double complex of sheaves. Let $E^{\bullet,\bullet}$ be the result of taking global sections: a double complex of vector spaces. Note that each column is a direct sum of complexes, according to the $\lambda$ summands in \eqref{eqn:fbullet-sum-decomp}.

Running the sequence with the horizontal maps first, it converges to 0 on the $E_2$ page (since $\mc{E} \otimes \ms{F}^\bullet$ is an exact sequence of sheaves). In particular, $\mathrm{Tot}(E^{\bullet,\bullet})$ is exact. Running the sequence with the vertical maps first, we instead obtain, on the $E_1$ page,
\[E_1^{p,q} = \bigoplus_\lambda H^q(\mathcal{E} \otimes \mathbb{S}_\lambda(\mathcal{S}))^{\beta_{p,\lambda}}.\]
Observe that the $\lambda$ summand has dimension $\beta_{p,\lambda}(F^\bullet) \gamma_{q,\lambda}(\mathcal{E})$. The $(i,\lambda)$ coefficient produced in the Boij-S\"{o}derberg pairing, $\widetilde{\phi_{i,\lambda}}(F^\bullet, \mathcal{E})$, is the sum of this quantity along the diagonal $\{p-q=i\}$. That is, $\widetilde{\Phi}$ is akin to a Betti table for $\mathrm{Tot}(E_1)$: \[\widetilde{\phi_{i,\lambda}} = \dim_\mb{C} \mathrm{Tot}(E_1)_{i,\lambda}.\]
We emphasize, however, that there is no actual $GL_k$-action on $\mr{Tot}(E_1)$, nor an $R_{k,k}$-module structure.

Instead, we will show by homological techniques that, for a wide class of double complexes including $E^{\bullet,\bullet}$, there is a perfect matching on a graph associated to $\mr{Tot}(E_1)$; in our setting, this will give the desired perfect matching on the Betti graph of $\widetilde{\Phi}$.
\end{proof}
\noindent The key properties of the double complex $E^{\bullet,\bullet}$ constructed above are that 
\begin{itemize}
\item[(1)] Each term $E^{p,q}$ has a direct sum decomposition labeled by a poset $P$;
\item[(2)] The vertical maps $d_v$ are label-preserving;
\item[(3)] The horizontal maps are strictly-label-decreasing (in our setting, this follows by minimality of $F^\bullet$ and functoriality of the Cech complex);
\end{itemize}
By (1) and (2), the $E_1$ page (the homology of $d_v$) again has a direct sum decomposition labeled by $P$, $E_1^{p,q} = \bigoplus_{\lambda \in P} E_1^{p,q,\lambda}$. We define the following graph:

\begin{definition} \label{def:e1-coeff-graph}
The \newword{$\mathbf{E_1}$ graph} $G = G(E^{\bullet,\bullet})$ is the following directed graph:
\begin{itemize}
\item The vertex set contains $\dim(E_1^{p,q,\lambda})$ vertices labeled $(p,q,\lambda)$, for each $p,q \in \mb{Z}$ and each $\lambda \in P$;
\item The edge set includes all possible edges $(p,q,\lambda) \to (p',q',\lambda')$ whenever $\lambda' \prec \lambda$ and $(p',q') = (p-r,q-r+1)$ for some $r > 0$.
\end{itemize}
The edges of $G$ are shaped like the higher-order differentials of the associated spectral sequence, and they respect the strictly-decreasing-$P$-labels condition.
\end{definition}
We show:
\begin{theorem} \label{thm:e1-perfect-matching}
Let $E^{\bullet,\bullet}$ be a double complex of vector spaces satisfying (1)-(3). If $\mr{Tot}(E^{\bullet,\bullet})$ is exact, its $E_1$ graph has a perfect matching.
\end{theorem}

We think of this theorem as a combinatorial analog of the fact that the associated spectral sequence (beginning with the homology of $d_v$) converges to zero. We explore this idea further in Section \ref{subsec:matchings}.

In our setting, we identify the vertices of the $E_1$ graph and the Betti graph of $\widetilde{\Phi}(\beta,\gamma)$; for any such identification, the edges of the $E_1$ graph become a subset of the Betti graph's edges. (We may recover the missing edges by allowing $r \leq 0$ in Definition \ref{def:e1-coeff-graph}.) Hence, the perfect matching produced by Theorem \ref{thm:e1-perfect-matching} is valid for the Betti graph, completing the proof of Theorem \ref{thm:pairing}.

\subsection{Perfect matchings in linear and homological algebra} \label{subsec:matchings}
Our approach uses linear maps to produce perfect matchings. The starting point is the following construction:

\begin{definition}
Let $T : V \to W$ be a map of vector spaces, having specified bases $\mc{V}, \mc{W}$. The \newword{coefficient graph} $G$ of $T$ is the directed bipartite graph with vertex set $\mc{V} \sqcup \mc{W}$ and edges
\[ E = \{v \to w : T(v) \text{ has a nonzero $w$-coefficient} \}.\]
Note that the adjacency matrix of $G$ is $T$ with all nonzero coefficients replaced by $1$'s.
\end{definition}

\begin{proposition} \label{prop:baby-matching}
For finite-dimensional vector spaces, the coefficient graph of an isomorphism admits a perfect matching.
\end{proposition}
We will say the corresponding bijection $\mc{V} \leftrightarrow \mc{W}$ is \newword{compatible with $T$}, a combinatorial analog of the fact that $T$ is an isomorphism. The proof of existence is simple, but essentially nonconstructive in practice. Here are two ways to do it:
\begin{enumerate}
\item[(i)] All at once: since $\det(T) \ne 0$, some monomial term of $\det(T)$ is nonzero. This exhibits the perfect matching.
\item[(ii)] By induction, using the Laplace expansion: expand $\det(T)$ along a row or column; some term $a_{ij} \cdot \text{(complementary minor)}$ is nonzero, and so on.
\end{enumerate}
Similarly, if $T$ is merely assumed to be injective or surjective, we may produce a maximal matching in this way (choose some nonvanishing maximal minor).

We generalize Proposition \ref{prop:baby-matching} to the setting of homological algebra in three ways: to infinite-dimensional vector spaces, to long exact sequences, and to double complexes (motivated by spectral sequences).

\begin{proposition} \label{prop:isom-matching}
For vector spaces of arbitrary dimension, the coefficient graph of an isomorphism admits a perfect matching.
\end{proposition}

We will not need Proposition \ref{prop:isom-matching} for our proof of Theorem \ref{thm:e1-perfect-matching}, so we prove it in the appendix.

\subsubsection{Long exact sequences and the proof of Theorem \ref{thm:derived-cone}} \label{subsubsec:les-matching}
We generalize to the case of long exact sequences. Let 
\[ \cdots \leftarrow V_i \xleftarrow{\delta} V_{i+1} \leftarrow \cdots\]
be a long exact sequence, with $\mc{V}_i$ a fixed basis for $V_i$. (The vector spaces may be finite- or infinite-dimensional.)
\begin{definition}
The \newword{coefficient graph} $G$ for $(V_\bullet,\delta)$ (with respect to $\mc{V}_\bullet$) is the directed graph with vertex set $\bigsqcup_i \mc{V}_i$ and an edge $v \to v'$ whenever $\delta(v)$ has a nonzero $v'$-coefficient.
\end{definition}
\begin{proposition} \label{prop:les-matching}
The coefficient graph of a long exact sequence has a perfect matching.
\end{proposition}

\begin{proof}
Choose subsets $\mc{F}_i \subset \mc{V}_i$ descending to bases of $\mathrm{im}(\delta) \subset V_{i-1}$, using Zorn's Lemma in the infinite case. Let $\mc{G}_i = \mc{V}_i - \mc{F}_i$, and let $F_i = \mathrm{span}(\mc{F}_i)$ and $G_i = \mathrm{span}(\mc{G}_i)$. The composition $\tilde{\delta} : F_{i+1} \hookrightarrow V_{i+1} \xrightarrow{\delta} V_i \twoheadrightarrow G_i$ is an isomorphism and has the same coefficients as $\delta$, restricted to $\mc{F}_{i+1}$ and $\mc{G}_i$. 
Thus Proposition \ref{prop:baby-matching} (or \ref{prop:isom-matching} in the infinite case) yields a matching of $\mc{F}_{i+1}$ with $\mc{G}_i$.
%
\end{proof}

At this point, we complete the proof of Theorem \ref{thm:derived-cone}, characterizing the derived Boij-S\"{o}derberg cone $\ercone{k,k}$ of the square matrices.
\begin{corollary} \label{cor:derived-cone-proof}
If $\widetilde{\beta} \in \ercone{k,k}$, then the Betti graph $G(\widetilde{\beta})$ has a perfect matching.
\end{corollary}

\begin{proof}
Let $\widetilde{\beta}$ be the Betti table of a minimal free equivariant complex $(F^\bullet,\delta)$ of $R$-modules, with $R = R_{k,k}$ the coordinate ring of the $k\times k$ matrices, and $F^\bullet \otimes R[\tfrac{1}{\det}]$ exact.

Choose, for each $F^i$, a $\mathbb{C}$-basis of each copy of $\mathbb{S}_\lambda(V)$ occuring in $F^i$. Label each basis element $x$ by the corresponding partition $\lambda$. It follows from minimality that $\delta(x_\lambda)$ is an $R$-linear combination of basis elements labeled by partitions $\lambda' \subsetneq \lambda$.

Since the homology modules are torsion, $F_\bullet \otimes \mathrm{Frac}(R)$ is an exact sequence of $\mr{Frac}(R)$-vector spaces, with bases given by the $x_\lambda$'s chosen above. By the previous proposition, its coefficient graph has a perfect matching. This graph has the same vertices as the Betti graph $G(\widetilde{\beta})$, and its edges are a subset of $G(\widetilde{\beta})$'s edges.
\end{proof}

\begin{remark}
Rather than tensoring with $\mathrm{Frac}(R)$, we may instead specialize to any convenient invertible $k \times k$ matrix $T \in \Hom(\mb{C}^k,\mb{C}^k)$, such as the identity matrix. This approach is useful for computations, since the resulting exact sequence consists of finite-dimensional $\mb{C}$-vector spaces.
\end{remark}

\subsubsection{Double complexes and the proof of Theorem \ref{thm:pairing}}

Finally, we generalize to the setting of double complexes and spectral sequences. Let $(E^{\bullet,\bullet}, d_v, d_h)$ be a double complex of vector spaces, with differentials pointing up and to the left:
\[\xymatrix{ 
&& E^{p-1,q+1} & E^{p,q+1} \ar[l]_-{d_h} \\
\ar@{-->}[u]^{q \text{ axis}} \ar@{-->}[r]_{p \text{ axis}} && E^{p-1,q} \ar[u]^-{d_v} & \ar[l]^-{d_h} \ar[u]_-{d_v} E^{p,q}
}\]
We assume the squares anticommute, so the total differential is
\[d_{tot} = d_h + d_v, \qquad \text{ and } \qquad  d_hd_v + d_vd_h = 0.\]
We will always assume the total complex $\mathrm{Tot}(E^{\bullet,\bullet})$ has a finite number of columns. Note that we do \emph{not} assume a basis has been specified for each $E^{\bullet,\bullet}$. We recall the complexes $E^{\bullet,\bullet}$ of interest:
\begin{itemize}
\item[(1)] Each term $E^{p,q}$ has a direct sum decomposition 
\[E^{p,q} = \bigoplus_{\lambda \in P} E^{p,q,\lambda},\]
with labels $\lambda$ from a poset $P$.
\item[(2)] The vertical differential $d_v$ is \emph{graded} with respect to this labeling, and
\item[(3)] The horizontal differential $d_h$ is \emph{downwards-filtered}.
\end{itemize}
The conditions (2) and (3) mean that
\[d_v(E^{p,q,\lambda}) \subseteq E^{p,q+1,\lambda}, \text{ and } d_h(E^{p,q,\lambda}) \subseteq \bigoplus_{\lambda' \prec \lambda} E^{p-1,q,\lambda'},\]
so the vertical differential preserves the label and the horizontal differential strictly decreases it.

We are interested in the homology of the vertical map $d_v$. Since $d_v$ is $P$-graded, so is its homology $E_1^{p,q,\lambda} = H(d_v)^{p,q,\lambda}$. We recall that the \newword{$\mathbf{E_1}$ graph} $G(E^{\bullet,\bullet})$ is defined as follows:
\begin{itemize}
\item The vertex set contains $\dim(E_1^{p,q,\lambda})$ vertices labeled $(p,q,\lambda)$, for each $p,q$ and each $\lambda \in P$;
\item The edge set includes all possible edges $(p,q,\lambda) \to (p',q',\lambda')$ whenever $\lambda' \prec \lambda$ and $(p',q') = (p-r,q-r+1)$ for some $r > 0$.
\end{itemize}
The edges of $G$ are shaped like higher-order differentials of the associated spectral sequence, i.e. they point downwards-and-leftwards, and they respect the strictly-decreasing-$P$-labels condition. We wish to show:
\begin{theorem}
If $\mr{Tot}(E^{\bullet,\bullet})$ is exact, the $E_1$ graph of $E^{\bullet,\bullet}$ has a perfect matching.
\end{theorem}

\begin{remark}
Consider summing the $E_1$ page along diagonals. Call the resulting complex $\mathrm{Tot}(E_1)$. If it were exact, the matching would exist by Proposition \ref{prop:les-matching}, and in fact would only use the edges corresponding to $r=1$. Since $\mathrm{Tot}(E_1)$ is not exact in general, the proof works by modifying its maps to make it exact.

Explicitly, we will exhibit a quasi-isomorphism from $\mathrm{Tot}(E^{\bullet,\bullet})$ to a complex with the same terms as $\mathrm{Tot}(E_1)$, but different maps -- whose nonzero coefficients are only in the spots permitted by the $E_1$ graph. Since $\mathrm{Tot}(E^{\bullet,\bullet})$ is exact, so is the new complex, so we will be done by Proposition \ref{prop:les-matching}.
\end{remark}

\begin{proof}
First, we split all the vertical maps: for each $p,q,\lambda$, we define subspaces $B,H,B^* \subseteq E$ (suppressing the indices) as follows. We put $B = \mathrm{im}(d_v)$; we choose $H$ to be linearly disjoint from $B$ and such that $B+H = \ker(d_v)$; then we choose $B^*$ linearly disjoint from $B+H$, such that $B+H+B^* = E$.

In particular, $d_v$ maps the subspace $B^*$ isomorphically to the subsequent subspace $B$, and the space $H$ descends isomorphically to $H(d_v)$, the $E_1$ term. The picture of a single column of the double complex looks like the following:
\[\hspace{2cm} \xymatrix@C-1.2pc{
& \vdots \\
B & H & B^* \ar[ull]_{\sim} & \text{ (note that } d_v(B) = d_v(H) = 0)\\
B & H & B^* \ar[ull]_{\sim} \\
   & H & B^* \ar[ull]_{\sim} \\
}\]
For the horizontal map, we have $d_h(B) \subset B$ and $d_h(H) \subset B+H$, and the poset labels $\lambda$ strictly decrease.

Our goal will be to choose bases carefully, so as to match the $H$ basis elements to one another, in successive diagonals, while decreasing the poset labels.

We first choose an arbitrary basis of each $H$ and $B^*$ space. We descend the basis of $B^*$ to a basis of the subsequent $B$ using $d_v$. Note that every basis element has a position $(p,q)$ and a label $\lambda$. We will write $x_\lambda$ if we wish to emphasize that a given basis vector $x$ has label $\lambda$.

We now change basis on the entire diagonal $E^i := \bigoplus_{p-q=i} E^{p,q}$. We leave the $H$ and $B^*$ bases untouched, but replace all the $B$ basis vectors, as follows. Let $b_\lambda \in B^{p,q,\lambda}$ and let $b^*_\lambda = d_v^{-1}(b_\lambda) \in (B^*)^{p,q-1,\lambda}$ be its `twin'. We define
\[\tilde{b}_\lambda := d_{tot}(b^*_\lambda) = b_\lambda + d_h(b^*_\lambda).\]
We replace $b_\lambda$ by $\tilde{b}_\lambda$, formally labeling the new basis vector by $(p,q,\lambda)$. We write $\widetilde{B}^{p,q,\lambda}$ for the span of these $\tilde{b}$'s, so in particular, $\widetilde{B}^{p,q,\lambda} := d_{tot}\big((B^*)^{p,q-1,\lambda}\big)$.

It is clear that $\widetilde{B},H,B^*$ collectively gives a new basis for the entire diagonal, unitriangular in the old basis. Notice also that the old basis element $b_\lambda \in B^{p,q,\lambda}$ becomes, in general, a linear combination of $\widetilde{B},H,B^*$ elements in all positions down-and-left of $p,q$, with leading term $\tilde{b}_\lambda$:
\[b_\lambda = \tilde{b}_\lambda + \sum_{i > 0} x^{p-i,q-i}, \text{ with } x^{p-i,q-i} \in \bigoplus_{\lambda' \subsetneq \lambda} E^{p-i,q-i,\lambda'}.\]
The lower terms have strictly smaller labels $\lambda' \subsetneq \lambda$. (In fact, slightly more is true: if a label $\lambda'$ occurs in the $i$-th term, the poset $P$ contains a chain of length $\geq i$ from $\lambda'$ to $\lambda$.)

We now inspect the coefficients of $(\mr{Tot}(E^{\bullet,\bullet}),d_{tot})$ in the new basis. We have
\begin{align*}
d_{tot}(b^*_\lambda) &= \tilde{b}_\lambda, \\
d_{tot}(\tilde{b}_\lambda) &= 0 \ \big(= d_{tot}^2(\tilde{b}^*_\lambda) \big),
\end{align*}
so the $B^*$ elements map one-by-one onto the $\widetilde{B}$ elements, with the same $\lambda$ labels; the latter elements then map to 0.

Next, for a basis element $h_\lambda \in H$, the coefficients change but remain `filtered'. If $d_{tot}(h_\lambda)$ included (in the old basis) some nonzero term $t \cdot b_\mu$, then in the new basis we have
\[d_{tot}(h_\lambda) = d_h(h_\lambda) = \cdots + t\cdot \big(\tilde{b}_\mu - d_h(b^*_\mu)\big) + \cdots.\]
Since $t$ is nonzero, we have $\mu \subsetneq \lambda$; and the additional terms coming from $d_h(b^*_\mu)$ all have labels $\mu' \subsetneq \mu$. Thus all labels occurring in $d_{tot}(h_\lambda)$ in the new basis are, again, strictly smaller than $\lambda$.  We note that $d_{tot}(h_\lambda)$ is a linear combination of $\widetilde{B},H,B^*$ elements down-and-left of $h_\lambda$ along the subsequent diagonal:
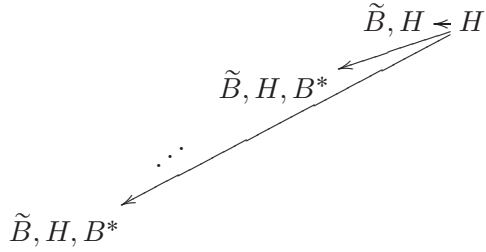
\begin{equation*}
\xymatrix@C-1.5pc@R-1.5pc{
&&& \widetilde{B},H & H \ar[l] \ar[dll] \ar[dddllll] \\
&& \widetilde{B},H,B^* \\
& \iddots \\
\widetilde{B},H,B^*
}
\end{equation*}
Finally, we observe that the spaces $\widetilde{B}+ B^*$ collectively span a subcomplex of $\mathrm{Tot}(E)$, so we have a short exact sequence of complexes
\[0 \to \mathrm{Tot}(\widetilde{B} + B^*) \to \mathrm{Tot}(E) \to \mathrm{Tot}(H) \to 0.\]
By construction, $\mathrm{Tot}(\widetilde{B} + B^*)$ is exact, so $\mathrm{Tot}(E) \to \mathrm{Tot}(H)$ is a quasi-isomorphism. Note that $\mathrm{Tot}(H)$ and $\mathrm{Tot}(E_1)$ have ``the same'' terms, but different maps, as desired. Since $\mathrm{Tot}(E)$ is exact, so is $\mathrm{Tot}(H)$. The desired matching therefore exists by Proposition \ref{prop:les-matching}.
\end{proof}

\begin{remark}
Our initial attempts to establish the Boij-S\"{o}derberg pairing (Theorems \ref{thm:pairing} and \ref{thm:e1-perfect-matching}) used the higher differentials on the $E_1, E_2, \ldots$ pages, rather than the $E_0$ page as above -- aiming to systematize ``chasing cohomology of the underlying sheaves''. 
The following example shows that such an approach fails on general double complexes.
\end{remark}

\begin{example}[A cautionary example]
Consider the following double complex. Each partition denotes a single basis vector with that label.
\[\xymatrix@C-1.2pc@R-1.2pc{& &\hspace{1.3cm}  {\tiny \yng(2)} & \ar[l] {\tiny \yng(3)} \\
& {\tiny \yng(2,1)} \oplus {\tiny \yng(1)} & \ar[l]_-{f} \ar@<-4ex>[u] {\tiny \yng(3,1)} \oplus {\tiny \yng(2)} \\
{\tiny \yng(1,1)} & \ar[l] {\tiny \yng(2,1)} \hspace{0.8cm} \ar@<3ex>[u]
}\]
The vertical map $d_v$ preserves labels and the horizontal map $d_h$ decreases labels. The unlabeled arrows correspond to coefficients of $1$, and the map $f$ is given by
\[f({\tiny \yng(2)}) = {\tiny\yng(1)}, \qquad f\Big({\tiny \yng(3,1)}\Big) = {\tiny\yng(2,1)} - {\tiny\yng(1)}.\]
Note that the rows are exact, so the total complex is exact as well, and the spectral sequence abuts to zero. The only nonzero higher differentials are on the $E_1$ and $E_3$ pages. These pages, and (for contrast) the complex $H$ constructed in Theorem \ref{thm:e1-perfect-matching}, are as follows.
\[\xymatrix@C-1.2pc@R-1.2pc{E_1: && 0 & {\tiny \yng(3)} \\
& {\tiny \yng(1)} & \ar[l] {\tiny \yng(3,1)}\\
{\tiny \yng(1,1)} & 0
} \qquad 
\xymatrix@C-1.2pc@R-1.2pc{E_3 : & & 0 & {\tiny \yng(3)} \ar@{-->}[ddlll] \\
& 0 & 0 \\
{\tiny \yng(1,1)} & 0
}\qquad \xymatrix@C-1.2pc@R-1.2pc{H : & & 0 & {\tiny \yng(3)} \ar[dll] \\
& {\tiny \yng(1)} & \ar[l] \ar[dll] {\tiny \yng(3,1)} \\
{\tiny \yng(1,1)} & 0
}\]
All the arrows are coefficients of $\pm 1$. In particular, \emph{no} combination of the $E_1$ and $E_3$ differentials gives a valid matching (the $E_3$ arrow violates the $P$-filtered condition). In contrast, $H$ finds the (unique) valid matching $\big\{\ {\tiny \yng(1)} \leftarrow {\tiny \yng(3)} \ , \ {\tiny \yng(1,1)} \leftarrow {\tiny \yng(3,1)}\ \big\}$.
\end{example}

\section{Preliminary results on \texorpdfstring{$2\times 3$}{2-by-3} matrices}
\label{sec:2x3mats}

While a complete picture of the general case still seems quite far away, we have some partial results and suggestive examples in the particular case of $2\times 3$ matrices, which suggest some features of the general picture.

In this setting, the modules of interest (as in Condition \ref{cond:modules-of-interest}) are Cohen-Macaulay of codimension 2. We will assume all modules are generated in positive degree. There are three equivariant Herzog-K\"{u}hl equations, corresponding to $\mu=\varnothing$, ${\scalebox{.5}{\yng(1)}}$, and $\scalebox{.5}{\yng(1,1)}$, which may be simplified to
\begin{align*}
\mu = \eset : \qquad  0 &= \sum_{i,\lambda} (-1)^i \beta_{i,\lambda} d_\lambda(2), \\
\mu = {\tiny \yng(1)} : \qquad 0 &= \sum_{i,\lambda} (-1)^i \beta_{i,\lambda} d_\lambda(2) \cdot (\lambda_1 + \lambda_2), \\
\mu = {\tiny \yng(1,1)} : \qquad 0 &= \sum_{i,\lambda} (-1)^i \beta_{i,\lambda} d_\lambda(2) \cdot \tfrac{1}{2}(\lambda_1+1)\lambda_2.
\end{align*}
All Betti tables in this section will be multiplicity Betti tables. We remark that
\begin{align*}
d_\lambda(2) &= 1+\lambda_1 - \lambda_2,\\
d_\lambda(3) &= \tfrac{1}{2} (1+\lambda_1 - \lambda_2)(1+\lambda_2-\lambda_3)(2+\lambda_1-\lambda_3).
\end{align*}

\subsection{Simple and pure resolutions}
The basic observation is the following. Since there are three Herzog-K\"{u}hl equations, if we allow exactly \emph{four} entries in our Betti table to be nonzero, in general we expect the equations to pick out one dimension's worth of valid tables. That is, the resulting table will be unique up to scalar multiple. Any realizable table of this form is automatically an extremal ray of $\mcone{2,3}$.

In fact, this observation underpins the characterization of pure tables for graded modules, where every choice of increasing degree sequence results in a unique table (up to scaling). The hope might be that, by analogy with both the graded case and the square-matrix case, these tables are all realizable, and form a complete set of extremal rays of $BS_{2,3}$. To that end, we will call these \newword{pure tables}.

For some choices of entries, the equations will be redundant, and allowing nonzero entries in only three positions will suffice. We will call the result a \newword{simple table}. In this case, each column of the table has exactly one entry, like the extremal tables in the graded and square-matrix cases. (Three entries are required or the resolution will be too short.)

\begin{remark}
In this section, most of the modules of interest are bi-equivariant for the actions of both $GL(V)$ and $GL(W)$ on $R = \Sym(\Hom(V,W)^*)$. So, for brevity, we write \[\lambda \otimes \mu := \mathbb{S}_\lambda(V) \otimes \mathbb{S}_\mu(W^*)\otimes R.\]
Note that this module has rank $d_\lambda(2)\cdot d_\mu(3)$, and therefore has multiplicity $d_\mu(3)$ if we remember only the action of $GL(V)$.
\end{remark}

\begin{example}
\label{exa:extremal-tables}
Here are all the pure tables using partitions of size between 0 and 3.

\begin{center}
\begin{tabular}{c|ccc}
&0&1&2\\\hline
$\varnothing$&3&&\\
${\scalebox{.5}{\yng(1)}}$&&3&\\
${\scalebox{.5}{\yng(2)}}$&&&1\\
${\scalebox{.5}{\yng(1,1)}}$&&&\\
${\scalebox{.5}{\yng(3)}}$&&&\\
${\scalebox{.5}{\yng(2,1)}}$&&&\\
\end{tabular}
\quad
\begin{tabular}{c|ccc}
&0&1&2\\\hline
$\varnothing$&8&&\\
${\scalebox{.5}{\yng(1)}}$&&6&\\
${\scalebox{.5}{\yng(2)}}$&&&\\
${\scalebox{.5}{\yng(1,1)}}$&&&\\
${\scalebox{.5}{\yng(3)}}$&&&1\\
${\scalebox{.5}{\yng(2,1)}}$&&&\\
\end{tabular}
\quad
\begin{tabular}{c|ccc}
&0&1&2\\\hline
$\varnothing$&1&&\\
${\scalebox{.5}{\yng(1)}}$&&&\\
${\scalebox{.5}{\yng(2)}}$&&&\\
${\scalebox{.5}{\yng(1,1)}}$&&3&\\
${\scalebox{.5}{\yng(3)}}$&&&\\
${\scalebox{.5}{\yng(2,1)}}$&&&1\\
\end{tabular}

\vspace{3ex}
\begin{tabular}{c|ccc}
&0&1&2\\\hline
$\varnothing$&2&&\\
${\scalebox{.5}{\yng(1)}}$&&&\\
${\scalebox{.5}{\yng(2)}}$&&2&\\
${\scalebox{.5}{\yng(1,1)}}$&&&\\
${\scalebox{.5}{\yng(3)}}$&&&1\\
${\scalebox{.5}{\yng(2,1)}}$&&&\\
\end{tabular}
\quad
\begin{tabular}{c|ccc}
&0&1&2\\\hline
$\varnothing$&&&\\
${\scalebox{.5}{\yng(1)}}$&6&&\\
${\scalebox{.5}{\yng(2)}}$&&8&\\
${\scalebox{.5}{\yng(1,1)}}$&&&\\
${\scalebox{.5}{\yng(3)}}$&&&3\\
${\scalebox{.5}{\yng(2,1)}}$&&&\\
\end{tabular}
\quad
\begin{tabular}{c|ccc}
&0&1&2\\\hline
$\varnothing$&&&\\
${\scalebox{.5}{\yng(1)}}$&3&&\\
${\scalebox{.5}{\yng(2)}}$&&1&\\
${\scalebox{.5}{\yng(1,1)}}$&&9&\\
${\scalebox{.5}{\yng(3)}}$&&&\\
${\scalebox{.5}{\yng(2,1)}}$&&&3\\
\end{tabular}
\end{center}
All of these tables except the last are simple, and each one is realizable. (The third is the Betti table of $M = R/P_2$, the quotient by the ideal of maximal minors, which is resolved by the Eagon-Northcott complex.)

In the graded and square-matrix cases, it is always possible \cite{EFW2011,FLS16} to realize any pure table by a resolution that is equivariant with respect to the actions of both $GL(V)$ and $GL(W)$. In fact, the same is true in these examples. The five simple tables above are realized, respectively, by resolutions of the following forms (the fourth actually gives 3 times the fourth table):
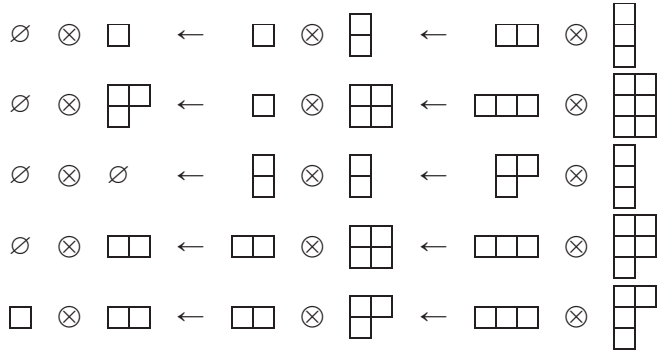
\[ \begin{array}{rclcrclcrcl} \vspace{0.1cm}
\eset & \otimes & {\tiny \yng(1)} & \leftarrow & {\tiny \yng(1)}& \otimes & {\tiny \yng(1,1)} & \leftarrow & {\tiny \yng(2)}& \otimes & {\tiny \yng(1,1,1)} \\ \vspace{0.1cm}
\eset & \otimes & {\tiny \yng(2,1)} & \leftarrow & {\tiny \yng(1)}& \otimes & {\tiny \yng(2,2)} & \leftarrow & {\tiny \yng(3)}& \otimes & {\tiny \yng(2,2,2)} \\ \vspace{0.1cm}
\eset & \otimes & \eset & \leftarrow & {\tiny \yng(1,1)}& \otimes & {\tiny \yng(1,1)} & \leftarrow & {\tiny \yng(2,1)}& \otimes & {\tiny \yng(1,1,1)} \\ \vspace{0.1cm}
\eset & \otimes & {\tiny \yng(2)} &\leftarrow & {\tiny \yng(2)}& \otimes & {\tiny \yng(2,2)} & \leftarrow & {\tiny \yng(3)}& \otimes & {\tiny \yng(2,2,1)} \\
{\tiny \yng(1)} & \otimes & {\tiny \yng(2)} &\leftarrow & {\tiny \yng(2)}& \otimes & {\tiny \yng(2,1)} & \leftarrow & {\tiny \yng(3)}& \otimes & {\tiny \yng(2,1,1)}
\end{array}\]
The last table is realised as follows:
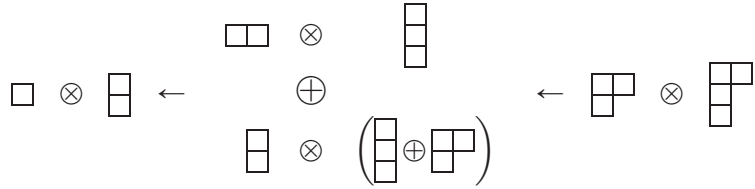
\[\begin{array}{rclcccrcl}
{\tiny \yng(1)} & \otimes & {\tiny \yng(1,1)} & \leftarrow &
\begin{array}{rcc} \vspace{0.1cm}
  {\tiny \yng(2)} & \otimes & {\tiny \yng(1,1,1)}\ \  \\ \vspace{0.1cm} & \bigoplus \\
  {\tiny \yng(1,1)} & \otimes & \bigg( {\tiny \yng(1,1,1)} \oplus {\tiny \yng(2,1)} \bigg)
\end{array}
 & \leftarrow & {\tiny \yng(2,1)}& \otimes & {\tiny \yng(2,1,1)}
\end{array}\]
\end{example}
One thing that stands out in the above example is that the biequivariant resolutions for each of the \emph{simple} tables has, in each step, a single irreducible $(GL_2\times GL_3)$-representation. This does not appear to be an accident: we can construct a large class of such tables using a technique similar to \cite[Theorem 4.6]{FLS16}.
%
The constructed tables are as follows. At each step of the resolution, the partition shape changes by the addition of a ``border strip,'' that is, a shape like
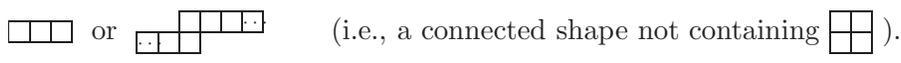
\[{\tiny \young(\hfil\hfil\hfil)}\ \text{ or }\ {\tiny \young(::\hfil\hfil\hfil\cdots,\cdots\hfil\hfil)} \qquad  \text{ (i.e., a connected shape not containing } {\tiny \yng(2,2)}\ ).\]
Moreover, the second border strip is always adjacent and to the right of the first.
\noindent Examples of such triples $(\lambda,\mu,\nu)$ are:
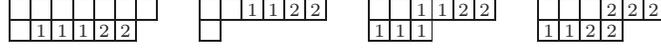
\[
{\tiny \young(\hfil\hfil\hfil\hfil\hfil\hfil\hfil,\hfil11122) \qquad \young(\hfil\hfil1122,\hfil) \qquad \young(\hfil\hfil1122,111) \qquad \young(\hfil\hfil\hfil222,1122)}
\]
Here $\lambda$ consists of the empty squares, $\mu$ contains the additional squares marked $1$, and $\nu$ contains the squares marked $1,2$.

\begin{proposition}
\label{prop:simple2by3}
Let $\beta$ be a simple table with entries $\lambda \subsetneq \mu \subsetneq \nu$. If $\mu/\lambda$ and $\nu/\mu$ are adjacent, successive border strips, then $\beta$ is realizable by a biequivariant resolution with an irreducible $GL_2 \times GL_3$ representation at each step.
\end{proposition}
We omit the proof. We can also rule out almost all other possible triples $(\lambda, \mu, \nu)$:
\begin{proposition}
Let $\beta$ be a simple table with entries $\lambda \subsetneq \mu \subsetneq \nu$. Then $\nu_2 \leq \lambda_1+1$.
\end{proposition}
That is, the shapes $\mu/\lambda$ and $\nu/\mu$ are \emph{contained in} the border strip formed by the squares along the outer edge of $\lambda$:
\[{\tiny \young(\hfil\hfil\hfil\hfil\hfil???,\hfil\hfil????) \ \raisebox{0.15cm}{$\cdots$} }\]
They need not be connected or adjacent, though computations have suggested that almost all are. We omit the proof, which follows from considering positive integer solutions to the equivariant Herzog-K\"{u}hl equations.

For tables that are pure but not simple, the story is currently much less complete, although we can realize many of them. Also, whether or not these tables are all realizable, we don't know if the cone has other extremal rays not of this form. Still, we have some data that at least suggests some helpful patterns. First, note that in every non-simple pure table, one column has two nonzero entries. We will say the table is \newword{diamond-shaped} when the middle column has two nonzero entries and \newword{Y-shaped} otherwise. In several cases, we can realize Y-shaped tables using extensions.

\begin{example}[Y-shaped tables via extensions]
\label{exa:extensions-1}
Consider the following Y-shaped table:

\begin{center}
\begin{tabular}{c|ccc}
&0&1&2\\\hline
${\scalebox{.5}{\yng(1)}}$&1&&\\
${\scalebox{.5}{\yng(2,1)}}$&&3&\\
${\scalebox{.5}{\yng(3,1)}}$&&&1\\
${\scalebox{.5}{\yng(2,2)}}$&&&1
\end{tabular}
\end{center}

Using the same notation as in the previous examples, we were able to realize this table by the following resolution:
\[ \begin{array}{rclcrclcrcl}
{\tiny \yng(1)} & \otimes & \eset & \leftarrow & {\tiny \yng(2,1)}& \otimes & {\tiny \yng(1,1)} & \leftarrow & \Big({\tiny \yng(2,2)} \oplus {\tiny \yng(3,1)} \Big)& \otimes & {\tiny \yng(1,1,1)}
\end{array}\]
It's possible to build this resolution out of two simple resolutions in the following way. Consider the following two simple resolutions, where we now explicitly name the modules being resolved:
\[ \begin{array}{ccrclcrclcrcl} \vspace{0.1cm}
M & \leftarrow & {\tiny \yng(1)} & \otimes & \eset & \xleftarrow{f} & {\tiny \yng(1,1)}& \otimes & {\tiny \yng(1)} & \xleftarrow{e} & {\tiny \yng(2,2)} & \otimes & {\tiny \yng(1,1,1)} \\
N & \xleftarrow{h} & {\tiny \yng(1,1)} & \otimes & {\tiny \yng(1)} & \xleftarrow{g} & {\tiny \yng(2,1)}& \otimes & {\tiny \yng(1,1)} & \leftarrow & {\tiny \yng(3,1)} & \otimes & {\tiny \yng(1,1,1)}\ .
\end{array}\]
Note that the second term in the resolution of $M$ is the same as the first term in the resolution of $N$, and that all the other terms appear in exactly the desired positions for the resolution we wish to construct.

We define $E := \coker(fg)$. By minimality, $E$ has the desired generators; it suffices to check that $E$ has the correct $\Tor_i$ for $i \geq 1$. The key fact is that there is a short exact sequence
\[0 \to N \to E \to M \to 0,\]
so we can extract information about the Betti table of $E$ from the long exact sequence in Tor. This follows essentially from computing that $h \circ e = 0$, that is, the unique generator of the module ${\tiny \yng(2,2)} \otimes {\tiny \yng(1,1,1)}$ (which has rank 1) maps into the image of $g$. Thinking of $\Ext^1(M,N)$ as obtained by applying $\Hom(-,N)$ to the resolution of $M$, the map $h$ therefore induces an extension class, which is $E$ above. 

Abusing notation, in the equation below, $\lambda \otimes \mu$ denotes the (finite-dimensional) representation $\mathbb{S}_\lambda(V) \otimes \mathbb{S}_\mu(W^*)$, not the corresponding free $R$-module. The long exact sequence in Tor is:
\begin{align*} \vspace{0.2cm}
0\  &\leftarrow {\tiny \yng(1)} \otimes \eset  \leftarrow  E \otimes R/\mf{m}  \leftarrow  {\tiny \yng(1,1)} \otimes {\tiny \yng(1)}  \xleftarrow{s}  {\tiny \yng(1,1)} \otimes {\tiny \yng(1)}  \leftarrow \Tor_1(E,R/\mf{m}) \leftarrow \\
&\leftarrow {\tiny \yng(2,1)} \otimes {\tiny \yng(1,1)} \xleftarrow{0} {\tiny \yng(2,2)} \otimes {\tiny \yng(1,1,1)} \leftarrow \Tor_2(E,R/\mf{m})\leftarrow {\tiny \yng(3,1)} \otimes {\tiny \yng(1,1,1)}\ \leftarrow \ 0.
\end{align*}
Since we have already determined $\Tor_0$, the map $s$ above is surjective, hence an isomorphism (since its source and target are the same irreducible representation). On the second line, the indicated map is zero because its source and target are distinct irreducible representations.

In this particular case, one can check by hand that $h \circ e = 0$ (in fact, we just computed the minimal free resolution of $\coker(fg)$ using Macaulay2 and checked that it had the desired form). It is not immediately clear why something similar should happen in general. Still, this cancellation seems important, and indeed shows up for many other Y-shaped pure tables.
\end{example}

We end by discussing two examples that demonstrate phenomena that do not appear in the square-matrix or graded cases.

\begin{example}[A slightly different extension]
\label{exa:extensions-2}
Consider the following table:

\begin{center}
\begin{tabular}{c|ccc}
&0&1&2\\\hline
${\scalebox{.5}{\yng(1)}}$&9&&\\
${\scalebox{.5}{\yng(2,1)}}$&&24&\\
${\scalebox{.5}{\yng(3,1)}}$&&&8\\
${\scalebox{.5}{\yng(3,2)}}$&&&3
\end{tabular}
\end{center}

A slight modification to the procedure of Example \ref{exa:extensions-1} realizes this table. We start with the following two resolutions:
\[ \begin{array}{ccrclcrclcrcl} \vspace{0.1cm}
M & \leftarrow & {\tiny \yng(1)} & \otimes & {\tiny \yng(1,1)} & \leftarrow & {\tiny \yng(1,1)}& \otimes & {\tiny \yng(2,1)} & \leftarrow & {\tiny \yng(3,2)} & \otimes & {\tiny \yng(2,2,2)} \\
N & \leftarrow & {\tiny \yng(1,1)} & \otimes & {\tiny \yng(1)} & \leftarrow & {\tiny \yng(2,1)}& \otimes & {\tiny \yng(1,1)} & \leftarrow & {\tiny \yng(3,1)} & \otimes & {\tiny \yng(1,1,1)}\ .
\end{array}\]
Note that $N$ is the same module as in Example \ref{exa:extensions-1}. This time, the $V$ sides line up (they both have a copy of ${\tiny \yng(1,1)}$ for $V$), but the $W$ sides do not. So, we tensor the resolution of $M$ with $\eset \otimes {\tiny \yng(1)}$ and the resolution of $N$ with $\eset \otimes {\tiny \yng(2,1)}$\ . Then they become composable, with the potential for cancellation. Continuing as in Example \ref{exa:extensions-1} (and assuming the requisite terms cancel, which they do) we get a resolution of the form
\[ \begin{array}{ccrclcrclcrcl} \vspace{0.1cm}
E & \leftarrow & {\tiny \yng(1)} & \otimes & \bigg( {\tiny \yng(1,1)} \otimes {\tiny \yng(1)} \bigg) & \leftarrow & {\tiny \yng(2,1)}& \otimes & \bigg( {\tiny \yng(1,1)} \otimes {\tiny \yng(2,1)} \bigg) & \leftarrow &
\begin{array}{rcl} \vspace{0.1cm}
{\tiny \yng(3,1)} & \otimes & {\tiny \yng(3,2,1)} \\ \vspace{0.1cm} & \bigoplus \\
{\tiny \yng(3,2)} & \otimes & {\tiny \yng(3,2,2)}\end{array}
\end{array},\]
which gives exactly the table above. In this case, not all of the $GL(W)$-representations are irreducible (though the ones on the final term are).
\end{example}

\begin{example}[``Stably-realizable'' Betti tables]
Consider the following table:

\begin{center}
\begin{tabular}{c|ccc}
&0&1&2\\\hline
$\varnothing$&1&&\\
${\scalebox{.5}{\yng(2)}}$&3&&\\
${\scalebox{.5}{\yng(2,1)}}$&&8&\\
${\scalebox{.5}{\yng(3,2)}}$&&&3
\end{tabular}
\end{center}
Despite satisfying the Herzog-K\"{u}hl equations, and resembling our previous unobjectionable examples, this table is not realizable. This follows from the numerical pairing: when paired with the trivial vector bundle, the result is the following table, which clearly has no perfect matching of the appropriate type:

\begin{center}
\begin{tabular}{c|cc}
&-1&0\\\hline
$\varnothing$&&1\\
${\scalebox{.5}{\yng(2)}}$&9&\\
${\scalebox{.5}{\yng(2,1)}}$&&8\\
\end{tabular}
\end{center}
\end{example}
Consider, however, tensoring the table with $\mathbb{S}_1V$. (That is, write the Betti table that would result from tensoring such a resolution with $\mathbb{S}_1V$.) The result is as follows:
\begin{center}
\begin{tabular}{c|ccc}
&0&1&2\\\hline
${\scalebox{.5}{\yng(1)}}$&1&&\\
${\scalebox{.5}{\yng(3)}}$&3&&\\
${\scalebox{.5}{\yng(2,1)}}$&3&&\\
${\scalebox{.5}{\yng(3,1)}}$&&8&\\
${\scalebox{.5}{\yng(2,2)}}$&&8&\\
${\scalebox{.5}{\yng(4,2)}}$&&&3\\
${\scalebox{.5}{\yng(3,3)}}$&&&3\\
\end{tabular}
\end{center} \vspace{0.1cm}
This table in fact \emph{is} realizable. It is a linear combination of the following four realizable tables:
\vspace{2ex}
\begin{center}
\begin{tabular}{c|ccc}
$A$&0&1&2\\\hline
${\scalebox{.5}{\yng(3)}}$&1&&\\
${\scalebox{.5}{\yng(3,1)}}$&&2&\\
${\scalebox{.5}{\yng(3,3)}}$&&&2\\
&&&
\end{tabular}
\quad
\begin{tabular}{c|ccc}
$B$&0&1&2\\\hline
${\scalebox{.5}{\yng(2,1)}}$&3&&\\
${\scalebox{.5}{\yng(3,1)}}$&&1&\\
${\scalebox{.5}{\yng(2,2)}}$&&6&\\
${\scalebox{.5}{\yng(4,2)}}$&&&1\\
\end{tabular}
\quad
\begin{tabular}{c|ccc}
$C$&0&1&2\\\hline
${\scalebox{.5}{\yng(3)}}$&1&&\\
${\scalebox{.5}{\yng(2,1)}}$&1&&\\
${\scalebox{.5}{\yng(3,1)}}$&&3&\\
${\scalebox{.5}{\yng(4,2)}}$&&&1\\
\end{tabular}
\quad
\begin{tabular}{c|ccc}
$D$&0&1&2\\\hline
${\scalebox{.5}{\yng(1)}}$&1&&\\
${\scalebox{.5}{\yng(2,2)}}$&&5&\\
${\scalebox{.5}{\yng(4,2)}}$&&&1\\
&&&
\end{tabular}
\end{center} \vspace{2ex}
Specifically, the large table above is \[\tfrac32 A+\tfrac12 B+\tfrac32 C+D.\]

This example suggests that it might be easier to study ``stably-realizable'' Betti tables, that is, Betti tables that become realizable after tensoring with some $GL(V)$-representation. This idea is still being explored.

\appendix
\section{The proof of Proposition \ref{prop:isom-matching}}

\begin{proposition}
Let $V,W$ be vector spaces of arbitrary dimension, with specified bases $\mc{V},\mc{W}$. Let $T : V \to W$ be an isomorphism. Then the coefficient graph of $T$ has a perfect matching.
\end{proposition}
\begin{proof}
We first reduce to $\mc{V},\mc{W}$ countable. This step is due to David Lampert \cite{L2016}. Let $b \in \mc{V}$ be arbitrary. Then $T(b)$ involves only finitely-many basis elements, say $B_1 \subset \mc{W}$. For each $s \in B_1$, $T^{-1}(s)$ only involves finitely-many basis elements; let $A_2 \subset \mc{V}$ contain these new elements, together with $b$. Repeat this construction, building two sequences of coordinate subspaces
\[(b) \subset A_2 \subset A_3 \subset \cdots \subset V, \qquad  B_1 \subset B_2 \subset B_3 \subset \cdots \subset W\]
such that, for each $i$, $B_i \subseteq T(A_i) \subseteq B_{i+1}.$ Let $A_\infty, B_\infty$ be the union; it follows that $T$ restricts to an isomorphism of countable-dimensional spaces $T : A_\infty \to B_\infty$, and $T$ as a whole splits as a direct sum of such isomorphisms.

We now build the matching inductively. Fix a basis vector $v \in \mc{V}$ and write $T(v) = \sum a_i w_i$, and assume every $w_i$ in the sum has $a_i \ne 0$. Equivalently,
\[v = \sum a_i T^{-1}(w_i),\]
so some $T^{-1}(w_i)$ contributes a nonzero $v$-coefficient. Fix one such $w$; we match $v \leftrightarrow w$. Note that this choice is compatible with both $T$ and $T^{-1}$. Let
\[C = \mathrm{span}( \mc{V} \setminus \{v\}), \qquad D = \mathrm{span}( \mc{W} \setminus \{w\}).\]
It is now easy to show that $C \hookrightarrow V \to W \twoheadrightarrow D$ is an isomorphism
%
(with the same coefficients as $T$, but with $v$ and $w$ removed). We now build the matching: we alternate between $V$ and $W$, always choosing the first unmatched basis vector on each side to ensure that every basis vector gets matched (note that the construction is symmetric with respect to $T$ and $T^{-1}$).
%
\end{proof}

\bibliographystyle{alpha}
\bibliography{boijsoder-bib}{}

\end{document}